\documentclass[a4paper,12pt]{amsart}
\usepackage{comment}

\usepackage{lipsum}
\usepackage{url}

\usepackage[english]{babel}
\usepackage[T1]{fontenc}

\usepackage[utf8]{inputenc}
\usepackage{amsmath,amssymb,mathrsfs}
\usepackage{amscd}
\usepackage{pgf,tikz}
\usetikzlibrary{arrows}
\usetikzlibrary{patterns}
\numberwithin{equation}{section}
\newtheorem{theo}{Theorem}[section]
\newtheorem{cor}{Corollary}[section]
\newtheorem{prop}{Proposition}[section]
\newtheorem{lem}{Lemma}[section]
\newtheorem{fe}{Definition}[section]
\newtheorem{rem}{Remark}[section]
\numberwithin{figure}{section}

\usepackage{a4wide}

\newcommand{\chh}{{H_V}}

\newcommand{\chho}{{H_0}}
\newcommand{\bn}{{\mathbb{N}}}
\newcommand{\br}{{\mathbb{R}}}
\newcommand{\bc}{{\mathbb{C}}}
\newcommand{\spp}{{\mathcal{S}_p}}
\newcommand{\sqq}{{\mathcal{S}_q}}
\newcommand{\s}{{\mathcal{S}_1}}
\newcommand{\sd}{{\mathcal{S}_2}}
\newcommand{\sinf}{{\mathcal{S}_\infty}}
\newcommand{\one}{{\bf 1}}

\newcommand{\xp}{x_\perp}
\newcommand{\mper}{m_\perp}
\newcommand{\im}{\textup{Im}}
\newcommand{\re}{\textup{Re}}

\title[Spectral Shift Function and Resonances]
{Spectral Shift Function and Resonances near the low 
ground state for Pauli and Schrödinger operators}
\author{Diomba \textsc{Sambou}}
\address{Departamento de Matem\'aticas, Facultad de Matem\'aticas, Pontificia Universidad
Cat\'olica de Chile, Vicu\~na Mackenna 4860, Santiago de Chile}
\email{disambou@mat.uc.cl}

\begin{document}


\begin{abstract}
We study the spectral shift function (SSF) $\xi(\lambda)$ and 
the resonances of the operator $H_V := \big( \sigma \cdot 
(-i\nabla - \textbf{A}) \big)^{2} + V$ in $L^2(\mathbb{R}^3)$ near 
the origin. Here $\sigma := (\sigma_1,\sigma_2,\sigma_3)$ are 
the $2 \times 2$ Pauli matrices and $V$ is a hermitian potential 
decaying exponentially in the direction of the magnetic field 
$\textbf{B} := \text{curl} \hspace{0.6mm} \textbf{A}$. We give 
a representation of the derivative of the SSF as a sum of the 
imaginary part of a holomorphic function and a harmonic measure 
related to the resonances of $H_V$. This representation warrant
the Breit-Wigner approximation moreover we deduce information
about the singularities of the SSF at the origin and a local 
trace formula.
\end{abstract}

\maketitle



\tableofcontents

\section{Introduction and motivations}\label{s1}

\subsection{Unperturbed operator}

Consider the three-dimensional Pauli 
operator acting in $L^{2}(\br^3) := L^{2}(\br^3,
\mathbb{C}^{2})$ and describing a quantum 
non-relativistic spin-$\frac{1}{2}$ particle 
subject to a magnetic field $\textbf{B} : 
\br^3 \longrightarrow \br^3$ pointing at
the $x_3$ direction: 
\begin{equation}\label{eq1,1}
\textbf{B}(\textbf{x}) = \big( 0,0,b(\textbf{x}) 
\big), \quad \textbf{x} := (x_{\perp},x_3) 
\in \br^{3}, \quad x_{\perp} := (x_1,x_2) 
\in \br^{2}.
\end{equation}
Then $\xp = (x_1,x_2) \in \mathbb{R}^{2}$ 
are the variables on the plane perpendicular 
to the magnetic field. Let $\textbf{A} = 
(a_{1},a_{2},a_{3}) : \mathbb{R}^{3} 
\rightarrow \mathbb{R}^{3}$ denote the 
magnetic potential generating the magnetic 
field, namely $\textbf{B}(\textbf{x}) := 
\text{curl} \hspace{0.6mm} \textbf{A}
(\textbf{x})$. Since $\text{div} 
\hspace{0.6mm} \textbf{B} = 0$ then $b$ is 
independent of $x_3$. Hence there is no 
loss of generality in assuming that 
$a_{j}$, $j = 1$, $2$ are independent 
of $x_3$ and $a_{3} = 0$: 
\begin{equation}\label{eq1,01}
\textbf{A}(\textbf{x}) 
= \big( a_1(\xp),a_2(\xp),0 \big),
\quad b(\textbf{x}) = b(\xp) = 
\partial_1 a_2(\xp) - \partial_2 a_1(\xp).
\end{equation}
Let $\sigma_{j}$, $j \in \lbrace 1, 2, 3 
\rbrace$ be the $2 \times 2$ Pauli 
matrices given by
\begin{equation}
\sigma_{1} := \begin{pmatrix}
   0 & 1 \\
   1 & 0
\end{pmatrix}, \hspace{0.5cm} 
\sigma_{2} := \begin{pmatrix}
   0 & -i \\
   i & 0
\end{pmatrix}, \hspace{0.5cm} 
\sigma_{3} := \begin{pmatrix}
   1 & 0 \\
   0 & -1
\end{pmatrix}.
\end{equation}
The free self-adjoint Pauli operator is initially defined 
on $C_{0}^{\infty}(\br^3,\mathbb{C}^{2})$ 
(then closed in $L^{2}(\br^3)$) by
\begin{equation}\label{eq1,2}
H_0 := \big( \sigma \cdot (-i\nabla - \textbf{A}) 
\big)^{2}, \qquad \sigma := (\sigma_1,\sigma_2,\sigma_3).
\end{equation}
A trivial computation shows that
\begin{equation}\label{eq1,3}
H_0 = \begin{pmatrix}
   (-i\nabla - \textbf{A})^{2} - b & 0 \\
   0 & (-i\nabla - \textbf{A})^{2} + b
\end{pmatrix}.
\end{equation}

We will assume (abusing the terminology) 
that $b : \br^{2} \rightarrow \br$ is an admissible 
magnetic field. This means that there exists a positive 
constant $b_{0}$ satisfying $b(\xp) = b_{0} +
\tilde{b}(\xp)$, $\tilde{b}$ being a function such 
that the Poisson equation 
\begin{equation}\label{eq1,4}
\Delta \tilde{\varphi} = \tilde{b}
\end{equation}
admits a solution $\tilde{\varphi} \in C^{2}(\br^{2})$ 
verifying $\sup_{\xp \in \br^{2}} \vert D^{\alpha} 
\tilde{\varphi}(\xp) \vert < \infty$, 
$\alpha \in \mathbb{Z}_{+}^{2}$, $\vert \alpha \vert 
\leq 2$, (we refer to \cite[Section 2.1]{rage} for more 
details and examples on admissible magnetic fields). 
Introduce $\varphi_{0}(\xp) = b_{0} \vert \xp \vert^{2}/4$ 
and $\varphi := \varphi_{0} + \tilde{\varphi}$ so that we 
have $\Delta \varphi = b$. Define originally on 
$C_{0}^{\infty}(\br^{2},\bc)$ the operators
\begin{equation}\label{eq1,5}
a = a(b) := -2i \textup{e}^{-\varphi} 
\frac{\partial}{\partial \bar{z}} \textup{e}^{\varphi} 
\quad \text{and} \quad a^{\ast} = a^{\ast}(b) := -2i 
\textup{e}^{\varphi} \frac{\partial}{\partial z} 
\textup{e}^{-\varphi}
\end{equation}
with $z := x_{1} + i x_{2}$, $\bar{z} := x_{1} - 
i x_{2}$ and introduce the operators
\begin{equation}\label{eq1,6}
H_{1}(b) = a^{\ast} a \quad \text{and} \quad \quad 
H_{2}(b) = a a^{\ast}.
\end{equation}
The spectral properties of $H_{j} = H_{j}(b)$, $j = 1$, 
$2$ are well know from \cite[Proposition 1.1]{rage}:
\begin{equation}\label{eq1,7}
\begin{cases}
\sigma (H_{1}) \subseteq \lbrace 0 \rbrace \cup 
[\zeta,+\infty) \hspace{0.1cm} \text{whith 
$0$ an eigenvalue of infinite multiplicity}, \\
\sigma (H_{2}) \subseteq  [\zeta,+\infty), \quad \dim 
\hspace{0.5mm} \textup{Ker} \hspace{0.5mm} H_{2} = 0,
\end{cases}
\end{equation}
where 
\begin{equation}\label{eq1,8}
\zeta := 2 b_{0} e^{-2 \hspace{0.5mm} \textup{osc} 
\hspace{0.5mm} \tilde{\varphi}}, \quad
\textup{osc} \hspace{0.5mm} \tilde{\varphi}:= 
\sup_{\xp \in \mathbb{R}^{2}} \tilde{\varphi} (\xp) 
- \inf_{\xp \in \br^{2}} \tilde{\varphi} (\xp).
\end{equation}
The orthogonal projection onto 
$\text{Ker} \hspace{0.5mm} H_{1}(b)$ will be denoted by 
$p = p(b)$. From \cite[Theorem 2.3]{hal} we know that 
it admits a continuous integral kernel 
$\mathcal{P}_{b}(\xp,\xp')$, $\xp$, $\xp' \in \br^{2}$.
Furthermore by \cite[Lemma 2.3]{rage}
\begin{equation}\label{eq1,81}
\frac{b_0}{2\pi} e^{-2\textup{osc} \hspace{0.5mm} 
\tilde{\varphi}} \leq \mathcal{P}_{b}(\xp,\xp) \leq 
\frac{b_0}{2\pi} e^{2\textup{osc} \hspace{0.5mm} 
\tilde{\varphi}}, \qquad \xp \in \br^2.
\end{equation}

Under the above considerations by taking
$a_1 = -\partial_2 \varphi$ and $a_2 = \partial_1 \varphi$ 
the operator $H_0$ can be written in $L^{2}(\br^{3}) 
= L^{2}(\br^{2}) \otimes L^{2}(\br)$ as
\begin{equation}\label{eq1,9}
\small{H_0 = \begin{pmatrix}
   H_{1}(b) \otimes 1 + 1 \otimes 
   \left( -\frac{d^{2}}{dx_3^{2}} \right) & 0 \\
   0 & H_{2}(b) \otimes 1 + 1 \otimes 
   \left( -\frac{d^{2}}{dx_3^{2}} \right)
\end{pmatrix} =: \begin{pmatrix}
   \mathcal{H}_{1}(b) & 0 \\
   0 & \mathcal{H}_{2}(b)
\end{pmatrix}}.
\end{equation}
The spectrum of  $-\frac{d^{2}}{dx_3^{2}}$ originally defined on 
$C_{0}^{\infty}(\br,\bc)$ coincides with $[0,+\infty)$ and is 
absolutely continuous. Then \eqref{eq1,7} and \eqref{eq1,9} 
imply that
\begin{equation}\label{eq1,10}
\sigma (H_0) = \sigma_{\textup{\textbf{ac}}} (H_0) = 
[0,+\infty),
\end{equation}
\big(see also \cite[Corollary 2.2]{rage}\big).

\subsection{Perturbed operator and the spectral shift function}

On the domain of $\chho$ we introduce the
perturbed operator 
\begin{equation}\label{eq1,11}
\chh := \chho + V,
\end{equation}
where $V$ is identified with the multiplication operator
by the matrix-valued function 
\begin{equation}\label{eq1,12}
V(\textbf{x}) := \begin{pmatrix}
   v_{11}(\textbf{x}) & v_{12}(\textbf{x}) \\
   v_{21}(\textbf{x}) & v_{22}(\textbf{x})
\end{pmatrix} \in \mathfrak{B}_h(\bc^2), \quad \textbf{x} 
\in \br^3,
\end{equation}
$\mathfrak{B}_h(\bc^2)$ being the set of $2 \times 2$ 
hermitian matrices. Throughout this paper we require an
exponential decay along the direction of the magnetic 
field for the electric potential $V$ in the following 
sense:
\begin{equation}\label{eq1,13} 
\begin{cases}
0 \not\equiv V \in C^0 (\br^3), \quad 
\vert v_{\ell k}(\textbf{x}) \vert \leq \text{Const.} 
\hspace{0.5mm} \langle \xp \rangle^{-\mper} 
\hspace{0.5mm} e^{-\gamma \langle
 x_3 \rangle}, \quad 1 \leq \ell,k \leq 2 \\
\textup{with $\mper > 2$, $\gamma > 0$ constant 
and $\langle y \rangle := \sqrt{1 + \vert y \vert^2}$
for $y \in \br^d$.}
\end{cases}
\end{equation}

Introduce some notations. Let $\mathscr{H}$ be 
a separable Hilbert space and $\sinf(\mathscr{H})$ be 
the set of compact linear operators on $\mathscr{H}$. 
Denote by $s_k(T)$ the $k$-th singular value of 
$T \in \sinf(\mathscr{H})$. The Schatten-von Neumann 
class ideals $\sqq(\mathscr{H})$, $q \in [1,+\infty)$ 
are defined by 
\begin{equation}\label{eq1,14}
\sqq(\mathscr{H}) := \Big\lbrace T \in \sinf(\mathscr{H}) 
: \Vert T \Vert^q_\sqq := \sum_k s_k(T)^q < +\infty 
\Big\rbrace.
\end{equation}
For $\lceil q \rceil := \min \big\lbrace n \in \mathbb{N} 
: n \geq q \big\rbrace$ and $T \in \sqq(\mathscr{H})$ the 
regularized determinant $\textup{det}_{\lceil q \rceil} 
(I - T)$ is defined by
\begin{equation}\label{eq1,15}
\small{\textup{det}_{\lceil q \rceil} (I - T)
:= \prod_{\mu \hspace*{0.1cm} \in \hspace*{0.1cm} \sigma (T)} 
\left[ (1 - \mu) \exp \left( \sum_{k=1}^{\lceil q \rceil-1} 
\frac{\mu^{k}}{k} \right) \right]}.
\end{equation}
The case $q = 1$ corresponds to the trace class operators
while the case $q = 2$ coincides with the Hilbert-Schmidt 
operators.

Now let $\mathcal{H}_0$ and $\mathcal{H}$ be two self-adjoint
operators in $\mathscr{H}$ such that
\begin{equation}\label{eq1,16}
V := \mathcal{H} - \mathcal{H}_0 \in \s(\mathscr{H}).
\end{equation}
There exists an important object in the theory of 
scattering associated to the pair of operators 
$(\mathcal{H},\mathcal{H}_0)$ called the \textit{spectral 
shift function} (SSF) $\xi(\lambda)$. The concept of 
SSF was first formally introduced by Lifshits \cite{lif}.
The mathematical theory of the SSF was developed by Krein 
\cite{kre1}. For trace class perturbations \eqref{eq1,16} 
the SSF is related to the determinant perturbation by the
Krein's formula \big(see for instance \cite{kre1}, \cite{kre2}\big)
\begin{equation}\label{eq1,17}
\xi(\lambda) = \frac{1}{\pi} 
\lim_{\varepsilon \longrightarrow 0^+} 
\text{Arg} \det 
\big( I + V(\mathcal{H}_0 - \lambda - i\varepsilon)^{-1} 
\big), \quad \text{a.e.} \hspace*{0.15cm} \lambda \in \br,
\end{equation}
the branch of the argument being fixed by the condition
$$
\text{Arg} \det 
\big( I + V(\mathcal{H}_0 - z)^{-1} \big)
\longrightarrow 0, \quad \im(z) \longrightarrow + \infty.
$$
Actually on the basis of the invariance principle 
\big(see for instance \cite{BiYa}\big) the SSF is well 
defined once there exists $\ell > 0$ such that
\begin{equation}\label{eq1,18}
(\mathcal{H} - i)^{-\ell} - (\mathcal{H}_0 - i)^{-\ell} 
\in \s(\mathscr{H}).
\end{equation}
It's the function whose derivative is given 
by the following distribution: 
\begin{equation}\label{eq1,19}
\xi' : f \longmapsto - \text{Tr} \big( f(\mathcal{H})
- f(\mathcal{H}_0) \big),  \quad 
f \in C_0^\infty(\br).
\end{equation}
Following the Birman-Krein theory \big(see \cite{BiKr}\big)
the SSF coincides with the scattering phase 
$s(\lambda) = -\frac{1}{2\pi} \text{Arg} \det 
S(\lambda)$ where $S(\lambda)$ is the scattering matrix.
More precisely by the Birman-Krein formula 
\big(see \cite{BiKr}\big) the SSF is related to $S(\lambda)$ 
by $\det S(\lambda) = e^{-2i\pi S(\lambda)}$ for almost 
every $\lambda \in \sigma_{ac}(\chho)$. The above
interpretation of the SSF as the scattering phase 
stimulates its investigation in quantum-mechanical 
problems. We refer to the review \cite{BiYa} and the 
book \cite{Yaf} for a large detailed bibliography about 
the SSF.

In our case assumption \eqref{eq1,13} on
$V$ implies that there exists $\mathscr{V} \in 
\mathscr{L}(\mathscr{H})$ such that
\begin{equation}\label{eq1,191}
\vert V \vert^\frac{1}{2} (\textbf{x})
= \mathscr{V} \left( \langle \xp \rangle
^{-\frac{m_\perp}{2}} \otimes
e^{-\frac{\gamma}{2} 
\langle t \rangle} \right), \quad 
\textbf{x} = (\xp,t) \in \br^3, \quad m_\perp > 2.
\end{equation} 
The standard criterion \cite[Theorem 4.1]{sim}
implies that
\begin{equation}\label{eq1,20}
\langle \xp \rangle
^{-\frac{m_\perp}{2}} \otimes
e^{-\frac{\gamma}{2} 
\langle t \rangle} (-\Delta + 1)^{-1} \in 
\sd \big( L^2(\br^3,\bc) \big).
\end{equation}
Then this together with the diamagnetic inequality 
\big(see \cite[Theorem 2.3]{avr}-\cite[Theorem 2.13]
{sim}\big) and the boundedness of the magnetic field $b$
imply that
\begin{equation}\label{eq1,21}
\vert V \vert^\frac{1}{2} (\chho - i)^{-1} \in 
\sd \big( L^2(\br^3) \big).
\end{equation}
Therefore exploiting the resolvent identity we obtain
\begin{equation}\label{eq1,22}
(\chh - i)^{-1} - (\chho - i)^{-1} 
\in \s \big( L^2(\br^3) \big).
\end{equation}
Namely \eqref{eq1,18} holds with $\ell = 1$ with respect 
to the operators $\chh$, $\chho$ and the Hilbert space 
$\mathscr{H} = L^2(\br^3)$. So the distribution 
\begin{equation}\label{eq1,23}
\xi' : f \longmapsto - \text{Tr} \big( f(\chh)
- f(\chho) \big),  \quad 
f \in C_0^\infty(\br)
\end{equation}
is well defined. For our purpose it is more convenient
to introduce the regularized spectral shift function
\big(see for instance \cite{kop} or \cite{bou}\big)
\begin{equation}\label{eq1,24}
\xi_2(\lambda) = \frac{1}{\pi} 
\lim_{\varepsilon \longrightarrow 0^+}
\text{Arg} \hspace*{0.1cm} {\det}_2 
\big( I + V(\chho - \lambda - i\varepsilon)^{-1} \big)
\end{equation}
whose derivative is given by the distribution
\begin{equation}\label{eq1,25}
\xi_2' : f \longmapsto - \text{Tr} \left( f(\chh)
- f(\chho) - \frac{d}{d\varepsilon} 
f(\chho + \varepsilon V)_{\vert \varepsilon = 0} \right), 
\quad f \in C_0^\infty(\br).
\end{equation}
From the relation between $\xi'$ and $\xi_2'$ given
by Lemma \ref{l5,1}, we will deduce the properties of 
the SSF. In the present paper the main result concerns a 
representation of the derivative of the SSF near the 
low ground state of the operator $\chho$ corresponding 
to the origin as a sum of a harmonic measure 
(related to the resonances of the operator $\chh$ near 
zero) and the imaginary part of a holomorphic function. 
Such representation justifies the Breit-Wigner 
approximation (see Theorem \ref{t2,1}) and implies a 
trace formula (see Theorem \ref{t2,2}) as in 
\cite{pet}, \cite{bru}, \cite{DiZe}, \cite{bon}. 
We derive also from our main 
result an asymptotic expansion of the SSF near the origin 
(see Theorem \ref{t2,3}). Similar results are obtained
in \cite{bon} for the SSF near the Landau levels as
well in \cite{kho}. On the other hand the singularities 
of the SSF associated to the pair $(\chh,\chho)$ is 
also studied in \cite{rage} with polynomial decay on 
the electric potential $V$. In Remark \ref{r2,2}, we compare 
our results to those of \cite{rage}. The case of the Dirac 
Hamiltonian with admissible magnetic fields is considered 
in \cite{tda} where the singularities of the SSF near 
$\pm m$ are investigated. Results obtained there are 
closely related to those from \cite{rage}.

The paper is organized as follows. In 
Section \ref{s2} we formulate our main
results. Sections \ref{s3}-\ref{s4} are
devoted to the study of the resonances
of $\chh$ near the origin. In the first 
one we define the resonances and in the 
second one we establish upper bounds on 
their number near the origin. Sections 
\ref{s5}-\ref{s7} are respectively 
devoted to the proofs of the main results.
Section \ref{sa} is a brief appendix on
finite meromorphic operator-valued 
functions.

\section{Statement of the main results}\label{s2}

First introduce some notations and terminology.

Denote by $\vert V \vert$ the multiplication operator 
by the matrix-valued function
\begin{equation}\label{eq2,1}
\sqrt{V^\ast V}(\textbf{x}) = \sqrt{V^2}(\textbf{x})
=: \big\lbrace \vert V 
\vert_{\ell k}(\textbf{x}) \big\rbrace, \quad 1 
\leq \ell,k \leq 2
\end{equation}
and by $J := sign(V)$ the matrix sign of $V$
which satisfies $V = J \vert V \vert$. We will say
that $V$ is of definite sign if the multiplication
operator $V(\textbf{x})$ by the matrix-valued function
$V(\textbf{x})$ satisfies 
\begin{equation}
\pm V(\textbf{x}) \geq 0
\end{equation} 
for any $\textbf{x} \in \br^3$. It is easy to 
check that in this case we have respectively $V 
= J \vert V \vert = \pm \vert V \vert$. Then 
without loss of generality we will say that $V$ is 
of definite sign $J = \pm$. 

Let $\textbf{W}$ be the 
multiplication operator by the function 
$\textbf{W} : \br^2 \longrightarrow 
\br$ defined by 
\begin{equation}\label{eq2,2}
\displaystyle \textbf{W}(\xp) :=
\int_\br \vert V \vert_{11} (\xp,x_3)dx_3.
\end{equation}
Hypothesis \eqref{eq1,13} on $V$ implies
that
\begin{equation}\label{eq2,3}
0 \leq \textbf{W} (\xp) \leq \text{Const.}' 
\hspace{0.5mm} \langle \xp \rangle^{-\mper}, 
\quad \mper > 2, \quad \xp \in \br^2,
\end{equation} 
where
$\text{Const.}' = \text{Const.} \hspace{0.5mm}
\int_\br e^{-\gamma \langle x_3 \rangle} dx_3$. 
Then by \cite[Lemma 2.3]{rage} the positive 
self-adjoint Toeplitz operator $p \textbf{W} p$ 
is of trace class, $p = p(b)$ being the 
orthogonal projection onto $\text{Ker} 
\hspace{0.5mm} H_{1}(b)$ defined by \eqref{eq1,6}. 

Introduce $e_\pm$ the multiplication operators by 
the functions $e^{\pm\frac{\gamma}{2} \langle \cdot
\rangle}$ respectively and let $c : L^{2}(\br) 
\longrightarrow \bc$ be the operator given by 
\begin{equation}\label{eqc1}
c(u) := \langle u,
e^{-\frac{\gamma}{2} \langle \cdot \rangle} 
\rangle
\end{equation}
while $c^{\ast} : \bc \longrightarrow 
L^{2}(\br)$ satisfies $c^{\ast}(\lambda) = 
\lambda e^{-\frac{\gamma}{2} \langle \cdot 
\rangle}$. Define the operator $K : L^{2}(\br^{3}) 
\longrightarrow L^{2}(\br^{2})$ by 
\begin{equation}\label{eq2,4}
K := \frac{1}{\sqrt{2}} (p \otimes c) 
\begin{pmatrix}
1 & 0 \\
0 & 0
\end{pmatrix}
e_+ \vert V \vert^{\frac{1}{2}}.
\end{equation} 
To be more explicit we have
\begin{equation}
(K \psi)(\textbf{x}) = \frac{1}{\sqrt{2}} 
\int_{\br^{3}}
{\mathcal P}_{b}(\xp,\xp^\prime) 
\begin{pmatrix}
1 & 0 \\
0 & 0
\end{pmatrix}
\vert V \vert^{\frac{1}{2}} (\xp^\prime,x_3^\prime) 
\psi (\xp^\prime,x_3^\prime)
d\xp^\prime dx_3^\prime,
\end{equation}
where ${\mathcal P}_{b}(\cdot,\cdot)$ is the integral 
kernel of the orthogonal projection $p$. Obviously the 
adjoint operator
$K^{\ast} : L^{2}(\br^{2}) \longrightarrow 
L^{2}(\br^{3})$
verifies
\begin{equation}
(K^{\ast}\varphi)(\xp,x_3) = 
\frac{1}{\sqrt{2}} \vert V \vert^{\frac{1}{2}} 
(\xp,x_3) 
\begin{pmatrix}
1 & 0 \\
0 & 0
\end{pmatrix}
(p \varphi)(\xp).
\end{equation}
Then
\begin{equation}\label{eq2,5}
K K^{\ast} = \begin{pmatrix}
1 & 0 \\
0 & 0
\end{pmatrix} 
\frac{p \textbf{\textup{W}} p}{2}
\end{equation}
so that it is a self-adjoint positif compact operator.

Now let us introduce technical important conditions.
Define the constant
\begin{equation}\label{eq2,51}
N_{\gamma,\zeta} := \min \left( \frac{\gamma}{2},
\sqrt{\zeta} \right),
\end{equation}
where $\gamma$ and $\zeta$ are respectively 
defined by \eqref{eq1,13} and \eqref{eq1,8}.
Let $\mathscr{W}_\pm \Subset \Omega_\pm$ be open 
relatively compact subsets of 
$\pm ]0,N_{\gamma,\zeta}^2[ e^{\pm i]-2\theta_0,2\varepsilon_0[}$
such that $0 < \min(\theta_0,\varepsilon_0)$ and 
$\max(\theta_0,\varepsilon_0) < \frac{\pi}{2}$. Let
$r > 0$ be a small parameter and assume that
$\mathscr{W}_\pm$ and $\Omega_\pm$ are simply connected
sets independent of $r$. We also assume that the 
intersections between $\pm ]0,N_{\gamma,\zeta}^2[$ and 
$\mathscr{W}_\pm$, $\Omega_\pm$ are intervals. Hence
we set $I_\pm := \mathscr{W}_\pm \cap \pm ]0,N_{\gamma,\zeta}^2[$.

In the case where the potential $V$ is of definite
sign $J = sign(V)$ the representation of the SSF near
zero can be specified. This required firstly that for 
$k \in \bc$ small enough the operator 
$I + \frac{iJ}{k} K^\ast K$ be 
invertible. That is for 
$\text{Arg} \hspace*{0.1cm} k \neq -J\frac{\pi}{2}$.
Secondly that the condition
\begin{equation}\label{eq2,6}
-J\frac{\pi}{2} \notin \left( \frac{\pi}{2} \right)_\mp 
\pm [-\theta_0,\varepsilon_0]
\end{equation} 
be satisfied with respect to the subscript "$\pm$"
in $\mathscr{W}_\pm \Subset \Omega_\pm$,
$I_\pm := \mathscr{W}_\pm \cap \pm ]0,N_{\gamma,\zeta}^2[$,
where $\left( \frac{\pi}{2} \right)_- = 0$ and 
$\left( \frac{\pi}{2} \right)_+ = \frac{\pi}{2}$.

\begin{rem}\label{r2,1}
$-$

\textbf{(i)} Under our considerations on 
$\theta_0$ and $\varepsilon_0$ above condition 
\eqref{eq2,6} is satisfied in the case "$+$" for 
$J = \pm$. Namely
\begin{equation}
\mp \frac{\pi}{2} \notin 
[-\theta_0,\varepsilon_0], \quad J = \pm.
\end{equation}

\textbf{(ii)} In the case "$-$" condition 
\eqref{eq2,6} is satisfied  for 
$J = +$. Namely
\begin{equation}
- \frac{\pi}{2} \notin 
\left[ \frac{\pi}{2} - \varepsilon_0,
\frac{\pi}{2} + \theta_0 \right].
\end{equation}
\end{rem}

From now on the set of the resonances near 
zero of $\chh$ (see Definition \ref{d3,1})
will be denoted by $\text{Res}(\chh)$.
Our first main result goes as follows:

\begin{theo}\label{t2,1}
$\textup{(Breit-Wigner approximation)}$

Assume that assumption \eqref{eq1,13} holds. 
Let $\mathscr{W}_\pm \Subset \Omega_\pm$ 
be open relatively compact subsets of 
$\pm ]0,N_{\gamma,\zeta}^2[ e^{\pm i]-2\theta_0,2
\varepsilon_0[}$ as above. Choose moreover 
$0 < s_1 < \sqrt{\textup{dist} \big( \Omega_\pm,0 
\big)}$. Then there exists $r_0 > 0$ and holomorphic 
functions $g_\pm$ in $\Omega_\pm$ satisfying for any
$\mu \in rI_\pm$ and $r < r_0$
\begin{equation}\label{eq2,7}
\xi'(\mu) = \frac{1}{r \pi} \im \hspace{0.5mm} g'_\pm 
\left( \frac{\mu}{r},r \right) + 
\sum_{\substack{w \in \textup{Res}(\chh) \cap 
r \Omega_\pm \\ \im (w) \neq 0}}
\frac{\im (w)}{\pi \vert \mu - w \vert^2}
- \sum_{w \in \textup{Res}(\chh) \cap 
r I_\pm} \delta (\mu - w),
\end{equation} 
where the functions $g_\pm(z,r)$ satisfy the bound
\begin{equation}\label{eq2,8}
\begin{split}
g_\pm(z,r) & = \mathcal{O} \left[ \textup{Tr} 
\hspace{0.4mm} \one_{(s_1\sqrt{r},\infty)} 
\big( p \textbf{\textup{W}} p \big)
\vert \ln r \vert + 
\Tilde{n}_1 \left( \frac{1}{2} s_1\sqrt{r} \right) + 
\Tilde{n}_2 \left( \frac{1}{2} s_1\sqrt{r} \right) 
\right] \\
& = \mathcal{O} \left( \vert \ln r \vert
r^{-1/m_\perp} \right),
\end{split}
\end{equation}
uniformly with respect to $0 < r < r_0$ and $z \in 
\Omega_\pm$, with $\Tilde{n}_q (\cdot)$, $q = 1$, $2$ 
defined by \eqref{eq4,17}.

Furthermore for potentials of definite sign 
$J = sign(V)$ we have for $\lambda \in rI_\pm$
\begin{equation}\label{eq2,9}
\frac{1}{r} \im \hspace{0.5mm} g'_\pm 
\left( \frac{\lambda}{r},r \right) = 
\frac{1}{r} \im \hspace{0.5mm} \Tilde{g}'_\pm 
\left( \frac{\lambda}{r},r \right) + 
\im \hspace{0.5mm} \Tilde{g}'_{1,\pm}(\lambda)
+ \one_{(0,N_{\gamma,\zeta}^2)}(\lambda) 
J \phi'(\lambda),
\end{equation}
where the function $\phi$ is defined by
\begin{equation}\label{eq2,10}
\phi(\lambda) := \textup{Tr} 
\hspace{0.4mm} \left( \arctan 
\frac{K^\ast K}{\sqrt{\lambda}} \right) 
= \textup{Tr} \hspace{0.4mm} \left( \arctan 
\frac{p\textbf{\textup{W}}p}{2\sqrt{\lambda}}
\right),
\end{equation}
the functions $z \mapsto \Tilde{g}_\pm (z,r)$ 
being holomorphic in $\Omega_\pm$ and satisfying
\begin{equation}\label{eq2,11}
\Tilde{g}_\pm (z,r) = \mathcal{O} 
\big( \vert \ln r \vert \big),
\end{equation}
uniformly with respect to $0 < r < r_0$ and 
$z \in \Omega_\pm$. The functions $z \mapsto 
\Tilde{g}_{1,\pm}(z)$ are holomorphic in 
$\pm ]0,N_{\gamma,\zeta}^2[ 
e^{\pm i]-2\theta_0,2\varepsilon_0[}$
and there exists a positive constant 
$C_{\theta_0}$ depending on $\theta_0$ such 
that
\begin{equation}\label{eq2,12}
\vert \Tilde{g}_{1,\pm}(z) \vert \leq
C_{\theta_0} \sigma_2 \left( \sqrt{\vert z \vert} 
\right)^{\frac{1}{2}}
\end{equation}
for $z \in \pm ]0,N_{\gamma,\zeta}^2[ 
e^{\pm i]-2\theta_0,2\varepsilon_0[}$, where the 
quantity $\sigma_2(\cdot)$ is defined by 
\eqref{eq4,15}.
\end{theo}

As first consequence of the above theorem we 
have the following result describing the 
asymptotic behaviour of the SSF on the right
of the low ground state.

\begin{theo}\label{t2,2}
$\textup{(Singularity at the low ground state)}$ 

Assume that $V$ satisfies assumption 
\eqref{eq1,13} with definite sign $J = 
sign(V)$. Then
\begin{equation}\label{eq2,13}
\xi(\lambda) = \frac{J}{\pi} \phi(\lambda) + 
\mathcal{O} \left( \phi(\lambda)^{\frac{1}{2}} 
\right) + \mathcal{O} \bigl( \vert \ln \lambda 
\vert^2 \bigr)
\end{equation}
as $\lambda \searrow 0$, the function 
$\phi(\lambda)$ being defined by 
\eqref{eq2,10}.
\end{theo}

\begin{rem}\label{r2,2}
$-$

\textbf{(i)} Since for $\lambda > 0$ 
\begin{equation}
\xi(-\lambda) = - \# \big\lbrace 
\text{discrete eigenvalues of $\chh$
lying in $(-\infty,-\lambda)$} \big\rbrace
\end{equation}
then for $V \geq 0$ we have 
$\xi(-\lambda) = 0$.

\textbf{(ii)} In \cite{rage} the 
singularities of the SSF near the origin 
are studied. If $\textbf{\textup{W}}$ 
satisfies assumptions \textbf{(A1)}, 
\textbf{(A2)} or \textbf{(A3)} implying 
respectively \eqref{eq02,2}, \eqref{eq02,3} 
or \eqref{eq02,4} then it is proved in
\cite{rage} that
\begin{equation}\label{eq2,131}
\xi(\lambda) = \frac{J}{\pi} \phi(\lambda)
\big( 1 + o(1) \big), \quad \lambda 
\searrow 0.
\end{equation}
Thus \eqref{eq2,13} provides a 
remainder estimate of \eqref{eq2,131} 
when $\textbf{\textup{W}}$ satisfies
assumption \textbf{(A1)}. However for 
$V \leq 0$ it is proved in \cite{rage} 
that
\begin{equation}
\xi(-\lambda) = -\textup{Tr} 
\hspace{0.4mm} \one_{(2\sqrt{\lambda},
\infty)} \big( p \textbf{\textup{W}} p 
\big) \big( 1 + o(1) \big), \quad 
\lambda \searrow 0.
\end{equation}
\end{rem}

As second consequence of Theorem \ref{t2,1} we 
have the following

\begin{theo}\label{t2,3}
\textup{(Local trace formula)} 

Let the domains $\mathscr{W}_\pm \Subset \Omega_\pm$ be 
as in Theorem \ref{eq2,1}. Assume that $f_\pm$ is 
holomorphic in a neighbourhood of $\Omega_\pm$
and let $\psi_\pm \in C_0^{\infty} \big( \Omega_\pm \cap 
\br \big)$ satisfy $\psi_\pm(\lambda) = 
1$ near $\Omega_\pm \cap \br$. Then under the assumptions 
of Theorem \ref{eq2,1}
\begin{equation}\label{eq2,14}
\textup{Tr} \hspace{0.4mm} \left[ (\psi_\pm f_\pm)
\left( \frac{\chh}{r} \right) - (\psi_\pm f_\pm)
\left( \frac{\chho}{r} \right) \right] =
\sum_{w \in \textup{Res}(\chh) \cap 
r \mathscr{W}_\pm} f_\pm \left( \frac{w}{r} \right)
+ E_{f_\pm,\psi_\pm}(r)
\end{equation}
with
\begin{equation}\label{eq2,15}
\vert E_{f_\pm,\psi_\pm}(r) \vert \leq M (\psi_\pm)
\sup \big\lbrace \vert f_\pm(z) \vert : z \in 
\Omega_\pm \setminus \mathscr{W}_\pm : \im(z) \leq 0
\big\rbrace \times N(r),
\end{equation}
where 
\begin{equation}\label{eq2,16}
\begin{split}
N(r) & = \textup{Tr} \hspace{0.4mm} 
\one_{(s_1\sqrt{r},\infty)} 
\big( p \textbf{\textup{W}} p \big)
\vert \ln r \vert + 
\Tilde{n}_1 \left( \frac{1}{2} s_1\sqrt{r} \right) + 
\Tilde{n}_2 \left( \frac{1}{2} s_1\sqrt{r} \right) 
\\
& = \mathcal{O} \left( \vert \ln r \vert
r^{-1/m_\perp} \right).
\end{split}
\end{equation}
\end{theo}

\begin{rem}
\textup{(Schrödinger operator)}

Our results remain true if instead 
the operator $\chh$ defined by \eqref{eq1,11}
we consider in $L^2(\br^3,\bc)$ the perturbed
Schrödinger operator
\begin{equation}
(-i\nabla - \textbf{A})^{2} - b + V 
\end{equation}
on $Dom \big( (-i\nabla - \textbf{A})^{2} - 
b \big)$ with $V(\textbf{\textup{x}}) = 
\mathcal{O} \left( \langle \xp \rangle^{-\mper} 
\hspace{0.5mm} e^{-\gamma \langle x_3 \rangle} 
\right)$ for any $\textbf{\textup{x}} \in \br^3$,
$m_\perp > 2$, $\gamma > 0$
as in \eqref{eq1,13}. Here $\textbf{\textup{W}}$ 
is just given by $\displaystyle 
\textbf{\textup{W}}(\xp) = \int_\br \vert V 
(\xp,x_3) \vert dx_3$ for any $\xp \in \br^2$ 
and in identities \eqref{eq2,4}-\eqref{eq2,5} 
the matrix 
$\begin{pmatrix}
1 & 0 \\
0 & 0
\end{pmatrix}$ is removed.
\end{rem}

\textbf{Acknowledgements}.
The author is partially supported by the Chilean 
Program \textit{N\'ucleo Milenio de F\'isica Matem\'atica
RC$120002$}. The author wishes to express his gratitude to
V. Bruneau for suggesting the study of this problem.

\section{Definition of the resonances}\label{s3}

The potential $V$ is assumed to satisfy 
\eqref{eq1,13}. We recall also that $p = p(b)$ is the 
orthogonal projection onto $\text{Ker} \hspace{0.5mm} 
H_{1}$ with $H_1 = H_1(b)$ defined by \eqref{eq1,6}.

Set $P := p \otimes 1$, 
$Q := I - P$. Introduce the orthogonal 
projections in $L^2(\br^3)$
\begin{equation}\label{eq3,1}
\textup{P} := \begin{pmatrix} 
   P & 0 \\ 
   0 & 0 
\end{pmatrix}, \hspace{1cm} 
\textup{Q} := \textup{I} - \textup{P} = 
\begin{pmatrix} 
   Q & 0 \\
   0 & I 
\end{pmatrix}.
\end{equation}
For $z \in \bc \setminus [0,+\infty)$
\eqref{eq1,11} and \eqref{eq1,7} imply that
\begin{equation}\label{eq3,2}
(\chho - z)^{-1} \textup{P} = 
\begin{pmatrix}
p \otimes \mathscr{R}(z) & 0 \\
   0 & 0
\end{pmatrix},
\end{equation}
where
$\mathscr{R}(z) := \left( -\frac{d^2}{dt^2} 
- z \right)^{-1}$ acts in $L^{2}(\br)$. 
Thus
\begin{equation}\label{eq3,3}
\big( \chho - z \big)^{-1} = 
\big( p \otimes \mathscr{R}(z) \big) 
\begin{pmatrix}
1 & 0\\
0 & 0
\end{pmatrix} + \big( \chho - 
z \big)^{-1} \textup{Q}.
\end{equation}
The one-dimensional resolvent $\mathscr{R}(z)$
introduced above admits the integral kernel
\begin{equation}\label{eq3,4}
\mathscr{N}_{z}(t,t') = 
\frac{i \textup{e}^{i \sqrt{z} \vert t - 
t' \vert}}{2 \sqrt{z}}
\end{equation}
if the branch $\im(\sqrt{z})$ is chosen such
that $\im(\sqrt{z}) > 0$. In the sequel we set
\begin{equation}\label{eq3,5}
\mathbb{C}^{+} := \big\lbrace z \in \bc:
\im(z) > 0 \big\rbrace \quad \text{and} \quad 
\bc_{1/2}^+ := \big\lbrace k \in 
\bc : k^{2} \in \bc^+ \big\rbrace.
\end{equation}
With respect to the variable $k$ we define 
the pointed disk
\begin{equation}\label{eq3,6}
D(0,\epsilon)^\ast := \big\lbrace k \in 
\mathbb{C} : 0 < \vert k \vert < \epsilon 
\big\rbrace
\end{equation}
with
\begin{equation}\label{eq3,7}
\epsilon < N_{\gamma,\zeta}
\end{equation}
the constant defined by \eqref{eq2,51}.

In order to define the resonances near zero
first we extend holomorophically 
$(H_{0} - k^2)^{-1} \textup{P}$ 
near $k = 0$.

\begin{prop}\label{p3,1}
Let $\gamma > 0$ be constant and set $z(k) := k^2$.

$\textup{\textbf{(i)}}$ The operator 
valued-function 
$$
k \longmapsto \left( \big( \chho - z(k) \big)
^{-1} \textup{P} : e^{-\frac{\gamma}{2} 
\langle t \rangle} L^{2}(\br^3) \longrightarrow 
e^{\frac{\gamma}{2} \langle t \rangle} 
L^{2}(\br^{3}) \right)
$$ 
admits a holomorphic extension from 
$\mathbb{C}_{1/2}^+ \cap D(0,\epsilon)^\ast$ 
to $D(0,\epsilon)^\ast$.

$\textup{\textbf{(ii)}}$ For $v_\perp(\xp) := 
\langle \xp \rangle^{-\alpha}$ with $\alpha > 1$
the operator valued-function
$$
T_{v_\perp} : k \longmapsto v_\perp(\xp) 
e^{-\frac{\gamma}{2} \langle t \rangle} 
\big( \chho - z(k) \big)^{-1} \textup{P} 
e^{-\frac{\gamma}{2} \langle t \rangle}
$$ 
has a holomorphic extension to 
$D(0,\epsilon)^{\ast}$ with values in the
Hilbert-Schmidt class 
$\sd \left( L^{2}(\mathbb{R}^{3}) \right)$.
\end{prop}

\begin{proof}
\textbf{(i)} Introduce
\begin{equation}
L(k) = \big{[} p \otimes \mathcal{R}(k^{2}) 
\big{]} \begin{pmatrix}
1 & 0\\
0 & 0
\end{pmatrix}
\end{equation}
acting from $e^{-\frac{\gamma}{2} 
\langle t \rangle} L^{2}(\mathbb{R}^{3})$ to 
$e^{\frac{\gamma}{2} \langle t \rangle} 
L^{2}(\mathbb{R}^{3})$. The operator 
$\mathscr{N}(k) := e^{-\frac{\gamma}{2} 
\langle t \rangle} \mathscr{R}(k^{2}) 
e^{-\frac{\gamma}{2} \langle t \rangle}$ 
admits the integral kernel
\begin{equation}\label{eq3,8}
e^{-\frac{\gamma}{2} \langle t \rangle} 
\frac{i \textup{e}^{i k \vert t - t' \vert}}
{2 k} \textup{e}^{-\frac{\gamma}{2} \langle 
t' \rangle}.
\end{equation}
It is easy to check that the integral
kernel \eqref{eq3,8} belongs to $L^{2}(\br)$ 
once $\im(k)> -\frac{\gamma}{2}$, $k \in \bc^\ast$. 
Then for $\epsilon < \frac{\gamma}{2}$ we can
extend holomorphically $k \longmapsto L(k) \in 
\mathscr{L} \left( e^{-\frac{\gamma}{2} 
\langle t \rangle} L^{2}(\br^{3}),
e^{\frac{\gamma}{2} \langle t \rangle} 
L^{2}(\br^{3}) \right)$ from $\bc_{1/2}^{+} 
\cap D(0,\epsilon)^{\ast}$ to 
$D(0,\epsilon)^\ast$. This together with
\eqref{eq3,2} imply that $k \longmapsto 
\big( \chho - z(k) \big)^{-1} \textup{P} 
\in \mathscr{L} \left( e^{-\frac{\gamma}{2} 
\langle t \rangle} L^{2}(\br^{3}),
e^{\frac{\gamma}{2} \langle t \rangle} 
L^{2}(\br^{3}) \right)$ admits a holomorphic
extension to $D(0,\epsilon)^{\ast}$.

\textbf{(ii)} Thanks to \eqref{eq3,2}
\begin{equation}\label{eq3,9}
T_{v_\perp}(k) = \left[ v_\perp p \otimes 
\mathscr{N}(k) \right] \begin{pmatrix}
1 & 0\\
0 & 0
\end{pmatrix}.
\end{equation}
The operator $\mathscr{N}(k) \in 
\sd \big( L^{2}(\br) \big)$ following 
the proof of assertion \textbf{(i)} for 
$\im(k) > -\frac{\gamma}{2}$, $k \in \bc^\ast$. 
Since $v_\perp^{2} \in L^1(\br^2)$ then by 
\cite[Lemma 2.3]{rage} $p v_\perp^{2} p$ 
is a trace class operator in $ L^{2}(\br^2)$. 
That is $v_\perp p v_\perp \in \s 
\big( L^{2}(\br^2) \big)$. This together
with \eqref{eq1,81} imply that
$v_\perp p \in \sd \big( L^{2}(\br^2) \big)$
with 
\begin{equation}\label{eq3,10}
\Vert v_\perp p \Vert_\sd^2 = \textup{Tr} 
\hspace{0.4mm} (v_\perp p v_\perp) =
\int_{\br^2} v_\perp^2(\xp)
\mathcal{P}_b (\xp,\xp) d\xp
\leq \frac{b_0}{2\pi} e^{2\textup{osc} 
\hspace{0.5mm} \tilde{\varphi}} 
\int_{\br^2} v_\perp^2(\xp) d\xp.
\end{equation}
Thus $k \mapsto T_{v_\perp}(k)$
has a holomorphic extension as above 
from $\bc_{1/2}^{+} \cap D(0,
\epsilon)^{\ast}$ to 
$D(0,\epsilon)^{\ast}$ with values in
$\sd \left( L^{2}(\br^{3}) \right)$.
The proof is complete.
\end{proof}

Now let us extend holomorphically the 
operator $(\chho - z)^{-1} \textup{Q}$ 
from the upper half-plane to the lower
half-plane except a semi-axis.

\begin{prop}\label{p3,2} 
Let $\gamma$ be as in Proposition 
\ref{p3,1} and $\zeta$ be defined by 
\eqref{eq1,8}.

\textup{\textbf{(i)}} The operator 
valued-function 
$$
z \longmapsto \left( (\chho - z)^{-1} 
\textup{Q} : e^{-\frac{\gamma}{2} \langle t 
\rangle} L^{2}(\br^{3}) \longrightarrow 
e^{\frac{\gamma}{2} \langle t \rangle} 
L^{2}(\br^{3}) \right)
$$ 
admits a holomorphic extension from $\bc^+$ 
to $\bc \setminus [\zeta,\infty)$.

\textup{\textbf{(ii)}} For $v_\perp(\xp) 
:= \langle \xp \rangle^{-\alpha}$ with 
$\alpha > 1$ the operator valued-function
$$
L_{v_\perp} : z \longmapsto v_\perp(\xp) 
e^{-\frac{\gamma}{2} \langle t \rangle} 
(\chho - z)^{-1} \textup{Q} 
e^{-\frac{\gamma}{2} \langle t \rangle}
$$ 
has a holomorphic extension to
$\bc \setminus [\zeta,\infty)$ with values 
in the Hilbert-Schmidt class 
$\sd \left( L^{2}(\mathbb{R}^{3}) \right)$.
\end{prop}

\begin{proof}
\textbf{(i)} Consider $z \in \bc^{+}$. Thanks 
to \eqref{eq1,9} and \eqref{eq3,1} we have 
\begin{equation}\label{eq3,11}
(\chho - z)^{-1} \textup{Q}
= \left( \begin{smallmatrix}
   \big( \mathcal{H}_1(b) - z \big)^{-1} 
   Q & 0 \\
   0 & \big( \mathcal{H}_2(b) - z 
   \big)^{-1}
\end{smallmatrix} \right) = 
\big( \mathcal{H}_1(b) - z \big)^{-1} Q 
\oplus \big( \mathcal{H}_2(b) - z 
\big)^{-1}.
\end{equation}
Since $\bc \setminus [\zeta,\infty)$ is 
contained in the resolvent set of
$\mathcal{H}_1(b)$ acting on 
$Q Dom \big( \mathcal{H}_1(b) \big)$ and 
$\mathcal{H}_2(b)$ acting on 
$Dom \big( \mathcal{H}_2(b) \big)$ then 
$\bc \setminus [\zeta,\infty) \ni z 
\longmapsto \big( \mathcal{H}_1(b) - z 
\big)^{-1} Q \oplus \big( 
\mathcal{H}_2(b) - z \big)^{-1}$ is well 
defined and holomorphic. So $\bc^+ \ni z 
\mapsto e^{-\frac{\gamma}{2} \langle t 
\rangle} (\chho - z)^{-1} \textup{Q} 
e^{-\frac{\gamma}{2} \langle t \rangle}$ 
admits a holomorphic extension to $\bc 
\setminus [\zeta,\infty)$.

\textbf{(ii)} According to \eqref{eq3,11}
\begin{equation}\label{eq3,110}
L_{v_\perp}(z) = v_\perp e^{-\frac{\gamma}{2} 
\langle t \rangle} \left( \big( 
\mathcal{H}_1(b) - z \big)^{-1} Q \oplus 
\big( \mathcal{H}_2(b) - z \big)^{-1} \right) 
e^{-\frac{\gamma}{2} \langle t \rangle}.
\end{equation}
We have
\begin{equation}\label{eq3,12}
\left\Vert v_\perp e^{-\frac{\gamma}{2} 
\langle t \rangle} \big( \mathcal{H}_1(b) - 
z \big)^{-1} Q \right\Vert_\sd^2 
\leq  \left\Vert v_\perp e^{-\frac{\gamma}{2} 
\langle t \rangle} \big( \mathcal{H}_1(b) + 1
\big)^{-1} \right\Vert_\sd^2
\left\Vert \big( \mathcal{H}_1(b) + 1 \big)
\big( \mathcal{H}_1(b) - z \big)^{-1} Q 
\right\Vert^2.
\end{equation}
By the Spectral mapping theorem
\begin{equation}\label{eq3,13}
\left\Vert \big( \mathcal{H}_1(b) + 1 \big)
\big( \mathcal{H}_1(b) - z \big)^{-1} Q 
\right\Vert^q \leq 
\textup{sup}_{s \in [\zeta,+\infty)}^q 
\left\vert \frac{s + 1}{s - z} \right\vert.
\end{equation}
With the help of the resolvent identity, the 
boundedness of the magnetic field $b$ and 
the diamagnetic inequality \big(see 
\cite[Theorem 2.3]{avr}-\cite[Theorem 2.13]
{sim}\big) we obtain
\begin{equation}\label{eq3,14}
\begin{split}
\left\Vert v_\perp e^{-\frac{\gamma}{2} 
\langle t \rangle} \big( \mathcal{H}_1(b) 
+ 1 \big)^{-1} \right\Vert_\sd^2 
& \leq \left\Vert I + \big( \mathcal{H}_1(b)
 + 1 \big)^{-1}b \right\Vert^2
\left\Vert v_\perp e^{-\frac{\gamma}{2} 
\langle t \rangle} 
\big( (-i\nabla - \textbf{A})^{2} + 1 
\big)^{-1} \right\Vert_\sd^2 \\
& \leq C \left\Vert v_\perp 
e^{-\frac{\gamma}{2} \langle t \rangle} 
(-\Delta + 1)^{-1} \right\Vert_\sd^2.
\end{split}
\end{equation}
By the standard criterion 
\cite[Theorem 4.1]{sim}
\begin{equation}\label{eq3,15}
\left\Vert v_\perp e^{-\frac{\gamma}{2} 
\langle t \rangle} (-\Delta + 1 \vert)^{-1} 
\right\Vert_\sd^2 \leq C 
\Vert v_\perp e^{-\frac{\gamma}{2} 
\langle t \rangle} \Vert_{L^2}^2 
\left\Vert \Bigl( \vert \cdot \vert^{2} 
+ 1 \Bigr)^{-1} \right\Vert_{L^2}^2.
\end{equation}
Putting together \eqref{eq3,12}, 
\eqref{eq3,13}, \eqref{eq3,14} and
\eqref{eq3,15} we get
\begin{equation}\label{eq3,16}
\left\Vert v_\perp e^{-\frac{\gamma}{2} 
\langle t \rangle} \big( \mathcal{H}_1(b) - 
z \big)^{-1} Q \right\Vert_\sd^2
\leq C \Vert v_\perp e^{-\frac{\gamma}{2} 
\langle t \rangle} \Vert_{L^{2}}^{2} 
\textup{sup}_{s \in [\zeta,+\infty)}^2 
\left\vert \frac{s + 1}{s - z} \right\vert.
\end{equation}
By similar arguments we can prove that
\begin{equation}\label{eq3,17}
\left\Vert v_\perp e^{-\frac{\gamma}{2} 
\langle t \rangle} \big( \mathcal{H}_2(b) - 
z \big)^{-1} \right\Vert_\sd^2
\leq C \Vert v_\perp e^{-\frac{\gamma}{2} 
\langle t \rangle} \Vert_{L^{2}}^{2} 
\textup{sup}_{s \in [\zeta,+\infty)}^2 
\left\vert \frac{s + 1}{s - z} \right\vert.
\end{equation}
Since the multiplication operator by the 
function $e^{-\frac{\gamma}{2} 
\langle t \rangle}$ is bounded then 
\eqref{eq3,110}, \eqref{eq3,16} and
\eqref{eq3,17} imply that $L_{v_\perp}(z)$ 
belongs to $\sd \left( L^{2}(\br^{3}) 
\right)$ and has a holomorphic extension 
from $\bc^{+}$ to $\bc \setminus [\zeta,
\infty)$. This completes the proof.
\end{proof}

For $V$ satisfying assumption 
\eqref{eq1,13}, \eqref{eq1,191} holds. 
Then this together with \eqref{eq3,3}, 
Propositions \ref{p3,1}-\ref{p3,2} 
yield to the following

\begin{lem}\label{l3,1} 
Let $D(0,\epsilon)^{\ast}$ be the 
pointed disk defined by \eqref{eq3,6}. 
Assume that $V$ satisfies \eqref{eq1,13} 
and set $z(k) := k^2$. Then the 
operator valued-function
$$
\mathbb{C}_{1/2}^{+} \cap D(0,\epsilon)
^{\ast} \ni k \longmapsto \mathcal{T}_{V}
\big( z(k) \big) := J \vert V \vert^{1/2} 
\big( \chho - z(k) \big)^{-1} \vert V 
\vert^{1/2},
$$ 
where $J := sign(V)$ has a holomorphic
extension to $D(0,\epsilon)^{\ast}$ 
with values in $\sd \left( L^{2}(\br^{3}) 
\right)$. We will denote again this 
extension by $\mathcal{T}_{V} \big( z(k)
\big)$. Furthermore the operator 
$\partial_z \mathcal{T}_{V} \big( z(k)
\big) \in \s \left( L^{2}(\br^{3}) 
\right)$ is holomorphic on 
$D(0,\epsilon)^{\ast}$.
\end{lem}

Now using the identity
$$
(\chh - z)^{-1} \big( 1 + V(\chho - 
z)^{-1} \big) = (\chho - z)^{-1}
$$
derived from the resolvent equation 
we obtain
\begin{align*}
e^{-\frac{\gamma}{2} \langle t \rangle} 
\big(\chh - z \big)^{-1} 
e^{-\frac{\gamma}{2} \langle t \rangle} 
& = e^{-\frac{\gamma}{2} \langle t 
\rangle} (\chho - z)^{-1} 
e^{-\frac{\gamma}{2} \langle t \rangle} 
\\ 
& \times \left( 1 + e^{\frac{\gamma}{2} 
\langle t \rangle} V(\chho - z)^{-1} 
e^{-\frac{\gamma}{2} \langle t \rangle} 
\right)^{-1}.
\end{align*}
As in \eqref{eq1,191} assumption 
\eqref{eq1,13} on $V$ implies the 
existence of
$\mathscr{M} \in \mathscr{L} \big( 
L^{2}(\br^3) \big)$ such that 
\begin{equation}\label{eq3,19}
\vert V \vert (\xp,t)
= \mathscr{M} \left( \langle \xp \rangle
^{-m_\perp} \otimes e^{-\gamma \langle t 
\rangle} \right), \quad (\xp,t) \in \br^3, 
\quad m_\perp > 2.
\end{equation} 
Then similarly to Lemma \ref{l3,1} it 
can be proved that $k \longmapsto 
e^{\frac{\gamma}{2} \langle t \rangle} 
V(\chho - z)^{-1} e^{-\frac{\gamma}{2} 
\langle t \rangle}$ is holomorphic with 
values in $\sinf \left( L^{2}(\br^{3}) 
\right)$. Thus by the analytic Fredholm 
theorem the operator valued-function 
$$
k \longmapsto \left( 1 + 
e^{\frac{\gamma}{2} \langle t \rangle} 
V(\chho - z)^{-1} e^{-\frac{\gamma}{2} 
\langle t \rangle} \right)^{-1}
$$ 
admits a meromorphic extension from 
$\mathbb{C}_{1/2}^{+} \cap 
D(0,\epsilon)^{\ast}$ to 
$D(0,\epsilon)^{\ast}$. Hence we have
the following

\begin{prop}\label{p3,3} 
Under the assumptions and the notations 
of Lemma \ref{l3,1} the operator 
valued-function 
$$
k \longmapsto \left( \big( \chh - z(k) 
\big)^{-1} : \textup{e}^{-\frac{\gamma}
{2} \langle x_{3} \rangle} L^{2}
(\mathbb{R}^{3}) \longrightarrow 
\textup{e}^{\frac{\gamma}{2} \langle 
x_{3} \rangle} L^{2}(\mathbb{R}^{3}) 
\right)
$$ 
admits a meromorphic extension from 
$\mathbb{C}_{1/2}^{+} \cap 
D(0,\epsilon)^{\ast}$ to 
$D(0,\epsilon)^{\ast}$. This extension 
will be denoted by $R \big( z(k) \big)$.
\end{prop}

We can now define the resonances of 
$\chh$ near zero. In the following 
definition the index of a 
finite-meromorphic operator 
valued-function appearing 
in \eqref{eq3,20} is recalled in the 
Appendix.

\begin{fe}\label{d3,1} 
We define the resonances of $H$ 
near zero as the poles of the 
meromorphic extension $R(z)$ of 
$\left( \chh - z \right)^{-1}$ in 
$\mathscr{L} \left( 
\textup{e}^{-\frac{\gamma}{2} 
\langle x_{3} \rangle} L^{2}(\br^3),
\textup{e}^{\delta \langle x_{3} 
\rangle} L^{2}(\br^3) \right)$. The 
multiplicity of a resonance $z_0 
:= z(k_0)$ is defined by
\begin{equation}\label{eq3,20}
\textup{mult}(z_0) := 
\textup{Ind}_{\mathcal{C}} 
\hspace{0.5mm} \Big( I + 
\mathcal{T}_{V} \big( z(\cdot) \big) 
\Big),
\end{equation}
$\mathcal{C}$ being a small 
contour positively oriented 
containing $k_0$ as the unique 
point $k \in D(0,\epsilon)^\ast$
satisfying $z(k)$ is a resonance
of $\chh$, and $\mathcal{T}_{V} \big( 
z(\cdot) \big)$ being defined by 
Lemma \ref{l3,1}.
\end{fe}

\section{Results on the resonances}\label{s4}

We establish preliminary 
results on the resonances we need 
for the proofs of our main results.

\subsection{A characterisation of the resonances}


\begin{prop}\label{p4,1} 
Let $\mathcal{T}_{V} (\cdot)$ be 
defined by Lemma \ref{l3,1}. Then the 
following assertions are equivalent:

$\textup{\textbf{(i)}}$ $z _0 := 
z(k_{0})$ is a resonance of $\chh$ 
near zero,

$\textup{\textbf{(ii)}}$ $-1$ is an 
eigenvalue of $\mathcal{T}_{V} \big( 
z(k_{0}) \big)$,

$\textup{\textbf{(iii)}}$ 
$\textup{det}_2 \big( I + 
\mathcal{T}_{V} \big( z(k_{0}) \big) 
\big) = 0$.
\\
Moreover the multiplicity of $z_0$ 
as zero of $\textup{det}_2 \big( I + 
\mathcal{T}_{V} (\cdot) \big)$ 
coincides with its multiplicity 
\eqref{eq3,20} as resonance of $\chh$.
\end{prop}

\begin{proof}
The equivalence \textbf{(i)} 
$\Leftrightarrow$ \textbf{(ii)} 
follows immediately from 
\begin{equation}\label{eq3,21}
\left( I + J \vert V \vert^{1/2} 
(\chho - z)^{-1} \vert V 
\vert^{1/2} \right) \left( I - J 
\vert V \vert^{1/2} (\chh - z)^{-1} 
\vert V \vert^{1/2} \right) = I.
\end{equation}

The equivalence \textbf{(ii)} 
$\Leftrightarrow$ \textbf{(iii)} is 
a direct consequence of the 
definition of $\textup{det}_2 
\bigl( I + \mathcal{T}_{V} \big( 
z(k_{0}) \big) \bigr)$ given by 
\eqref{eq1,16} with $q = 2$.

Otherwise since by Lemma \ref{l3,1}
$\mathcal{T}_{V} (\cdot)$ is 
holomorphic on 
$D(0,\epsilon)^{\ast}$ then so is 
$\textup{det}_2 \bigl( I + 
\mathcal{T}_{V} (\cdot) 
\big)$ on $D(0,\epsilon)^\ast$. Let 
$\textup{m}(z_0)$ be the multiplicity 
of $z_0$ as zero of
$\textup{det}_2 \big( I + 
\mathcal{T}_{V} (\cdot) 
\big)$. If $\mathcal{C}'$ is a 
small contour positively oriented 
containing $z_0$ as the unique 
resonance of $\chh$ near zero then
\begin{equation}\label{eq3,22}
\textup{m}(z_0) = ind_{\mathcal{C}'} 
\Big( \textup{det}_2 \bigl( I + 
\mathcal{T}_{V} (\cdot) 
\bigr) \Big),
\end{equation}
where the RHS of \eqref{eq3,22} is 
the index defined by \eqref{eqa,5} 
of the holomorphic function 
$\textup{det}_2 \bigl( I + 
\mathcal{T}_{V} (\cdot) 
\bigr)$ with respect 
to the contour $\mathcal{C}'$. Now 
the equality on the multiplicities 
claimed in the proposition is an 
immediate consequence of the 
equality
$$
ind_{\mathcal{C}'} 
\Big( \textup{det}_2 \bigl( I + 
\mathcal{T}_{V} (\cdot) 
\bigr) \Big) = Ind_{\mathcal{C}} 
\hspace{0.5mm} \Big( I + 
\mathcal{T}_{V} \big( z(\cdot) 
\big) \Big),
$$
see for instance \cite[(2.6)]{bo}. 
This concludes the proof.
\end{proof}

\subsection{Decomposition of the weighted resolvent}

We split the weighted resolvent
$\mathcal{T}_{V} \big( z(k) \big) :=  
J \vert V \vert^{\frac{1}{2}} 
\big( \chho - z(k) \big)^{-1} \vert V 
\vert^{\frac{1}{2}}$ into a singular 
part near $k = 0$ and a holomorphic 
part on the open disk $D(0,\epsilon) 
:= D(0,\epsilon)^\ast \cup \lbrace 0 
\rbrace$ with values in $\sd 
\big( L^{2}(\br^3) \big)$.

According to \eqref{eq3,3} for 
$k \in D(0,\epsilon)^\ast$
\begin{equation}\label{eq4,1}
\mathcal{T}_{V} \big( z(k) \big) = 
J \vert V \vert^{\frac{1}{2}} p 
\otimes \mathscr{R} \big( z(k) \big)
\begin{pmatrix}
1 & 0 \\
0 & 0
\end{pmatrix} \vert V 
\vert^{\frac{1}{2}} + J \vert V 
\vert^{\frac{1}{2}} \big( \chho - 
z(k) \big)^{-1} \textup{Q} \vert V 
\vert^{\frac{1}{2}}.
\end{equation}
Recall that $e_\pm$ are the 
multiplications operators by the 
functions $e^{\pm\frac{\gamma}{2} 
\langle \cdot \rangle}$ 
respectively. We have
\begin{equation}\label{eq4,2}
J \vert V \vert^{\frac{1}{2}} 
p \otimes \mathscr{R} \big( z(k) 
\big)
\begin{pmatrix}
1 & 0 \\
0 & 0
\end{pmatrix} \vert V 
\vert^{\frac{1}{2}}
= J \vert V \vert^{\frac{1}{2}} 
e_+ p \otimes e_- \mathscr{R} 
\big( z(k) \big) e_-
\begin{pmatrix}
1 & 0 \\
0 & 0
\end{pmatrix} e_+ \vert V 
\vert^{\frac{1}{2}}.
\end{equation}
Thanks to \eqref{eq3,4} the 
integral kernel of 
$e_- \mathscr{R} \big( z(k) \big) 
e_-$ is given by 
\begin{equation}\label{eq4,3}
e^{-\frac{\gamma}{2} \langle t 
\rangle} \frac{i \textup{e}^{i k 
\vert t - t' \vert}}{2 k}
e^{-\frac{\gamma}{2} \langle t' 
\rangle}
\end{equation}
for $k \in D(0,\epsilon)^\ast$. 
Then $e_- \mathscr{R} 
\big( z(k) \big) e_-$ can be 
decompose as
\begin{equation}\label{eq4,4}
e_- \mathscr{R} \big( z(k) \big) e_- 
= \frac{1}{k}a + b(k),
\end{equation}
where $a : L^{2}(\mathbb{R}) 
\longrightarrow L^{2}(\mathbb{R})$ 
is the rank-one operator defined by
\begin{equation}\label{eq4,5}
a(u) := \frac{i}{2} \big\langle u,
e^{-\frac{\gamma}{2} \langle \cdot 
\rangle} \big\rangle 
e^{-\frac{\gamma}{2} \langle \cdot 
\rangle}
\end{equation}
and $b(k)$ is the operator with 
integral kernel given by
\begin{equation}\label{eq4,6}
e^{-\frac{\gamma}{2} \langle t 
\rangle} i \frac{ \textup{e}^{i k 
\vert t - t' \vert} - 1}{2 k} 
e^{-\frac{\gamma}{2} \langle t 
\rangle}.
\end{equation}
It is easy to remark that $-2ia = 
c^\ast c$ where $c$ is the operator
defined by \eqref{eqc1}.
This together with \eqref{eq4,4} 
yield for $k \in D(0,\epsilon)^\ast$ 
to
\begin{equation}\label{eq4,7}
p \otimes e_- \mathscr{R} \big( z(k) 
\big) e_- = \pm \frac{i}{2k} p 
\otimes c^\ast c + p \otimes s(k),
\end{equation}
where $s(k)$ is the operator acting 
from $e^{-\frac{\gamma}{2} \langle t
\rangle} L^{2}(\br)$ to 
$e^{\frac{\gamma}{2} \langle t 
\rangle} L^{2}(\br)$ having the 
integral kernel
\begin{equation}\label{eq4,8}
\frac{ 1 - \textup{e}^{i k 
\vert t - t' \vert}}{2 i k}.
\end{equation}
By combining \eqref{eq4,2} with
\eqref{eq4,7} we get for $k \in 
D(0,\epsilon)^\ast$
\begin{equation}\label{eq4,9}
\begin{split}
& J
\vert V \vert^{\frac{1}{2}} 
p \otimes \mathscr{R} \big( z(k) 
\big)
\begin{pmatrix}
1 & 0 \\
0 & 0
\end{pmatrix} \vert V 
\vert^{\frac{1}{2}} \\
& = \frac{iJ}{2k} \vert V 
\vert^{\frac{1}{2}} 
e_+ (p \otimes c^\ast c) 
\begin{pmatrix}
1 & 0 \\
0 & 0
\end{pmatrix} e_+ \vert V 
\vert^{\frac{1}{2}} 
+ J \vert V \vert^{\frac{1}{2}} 
e_+ p \otimes s(k) 
\begin{pmatrix}
1 & 0 \\
0 & 0
\end{pmatrix}  e_+ 
\vert V \vert^{\frac{1}{2}}.
\end{split}
\end{equation}
That is
\begin{equation}\label{eq4,10}
J \vert V \vert^{\frac{1}{2}} 
p \otimes \mathscr{R} \big( z(k) \big)
\begin{pmatrix}
1 & 0 \\
0 & 0
\end{pmatrix} \vert V 
\vert^{\frac{1}{2}} = 
\frac{iJ}{k} K^\ast K + J
\vert V \vert^{\frac{1}{2}} e_+ p 
\otimes s(k) 
\begin{pmatrix}
1 & 0 \\
0 & 0
\end{pmatrix}
e_+ \vert V \vert^{\frac{1}{2}},
\end{equation}
where $K$ is the operator defined 
by \eqref{eq2,4}. We have then proved 
the following

\begin{prop}\label{p4,2} 
Let $V$ satisfy assumptions 
\eqref{eq1,12}-\eqref{eq1,13}. 
For $k \in D(0,\epsilon)^\ast$
\begin{equation}\label{eq4,11}
\mathcal{T}_{V} \big( z(k) \big) = 
\frac{iJ}{k} \mathscr{B}
+ \mathscr{A}(k), \quad \mathscr{B} 
:= K^\ast K,
\end{equation}
the operator $\mathscr{A}(k) \in 
\sd \big( L^{2}(\br^3) \big)$ being 
given by
\begin{equation}\label{eq4,111}
\mathscr{A}(k) := 
J \vert V \vert^{\frac{1}{2}} e_+ p 
\otimes s(k) 
\begin{pmatrix}
1 & 0 \\
0 & 0
\end{pmatrix}
e_+ \vert V \vert^{\frac{1}{2}}
 + J \vert V \vert^{\frac{1}{2}} 
\big( \chho - z(k) \big)^{-1} 
\textup{Q} \vert V \vert^{\frac{1}{2}}
\end{equation}
and holomorphic on the open disk 
$D(0,\epsilon)$ with $s(k)$ defined 
by \eqref{eq4,7}.
\end{prop} 

\begin{rem}
$-$

For any $r > 0$ we have
\begin{equation}\label{eq4,12}
\textup{Tr} \hspace{0.4mm} 
\one_{(r,\infty)} 
\left( K^\ast K \right) 
= \textup{Tr} \hspace{0.4mm} 
\one_{(r,\infty)} 
\left( K K^\ast \right)
= \textup{Tr} \hspace{0.4mm} 
\one_{(r,\infty)} 
\big( p \textbf{\textup{W}} p 
\big)
\end{equation}
following \eqref{eq2,5}.
\end{rem}

Note that the asymptotic 
expansion of the quantity 
$\textup{Tr} \hspace{0.4mm} 
\one_{(r,\infty)} 
\big( p U p \big)$ is well known 
once the function 
$0 \leq U \in L^\infty (\br^2)$ 
decays like a power, exponentially 
or is compactly supported:

\textbf{(A1)} If $U \in C^{1} 
\big( \mathbb{R}^{2} \big)$ satisfies
$U(\xp) = u_{0}\big(\xp / \vert \xp 
\vert\big) \vert \xp \vert^{-m} 
( 1 + o(1) \big)$, $\vert \xp \vert 
\rightarrow \infty$, $0 \not\equiv 
u_{0} \in C^0 \big( \mathbb{S}^{1},
\br_+ \big)$, $\vert \nabla U(\xp) 
\vert \leq C_{1} \langle \xp 
\rangle^{-m-1}$ with $m$, $C_{1} > 0$ 
constant and if there exists an 
integrated density of states for the 
operator $H_{1}(b)$ then
\begin{equation}\label{eq02,2}
\textup{Tr} \hspace{0.4mm} 
\one_{(r,\infty)} \big( p U p \big) 
= C_{m} r^{-2/m} \big( 1 + o(1) \big), 
\hspace{0.2cm} r \searrow 0,
\end{equation}
where $C_{m} := \frac{b_0}{4\pi} 
\int_{\mathbb{S}^{1}} u_{0}(t)^{2/m} 
dt$, \big(see \cite[Lemma 3.3]{rage}\big).

\textbf{(A2)} If $U$ 
satisfies $\ln U(\xp) = -\mu \vert 
\xp \vert^{2\beta} \big( 1 + o(1) 
\big)$, $\vert \xp \vert \rightarrow 
\infty$ with $\beta$, $\mu > 0$ 
constant then
\begin{equation}\label{eq02,3}
\textup{Tr} \hspace{0.4mm} 
\one_{(r,\infty)} \big( p U p \big) 
= \varphi_{\beta}(r) \big( 1 + o(1) 
\big), \hspace{0.2cm} r \searrow 0,
\end{equation}
where for $0 < r < \textup{e}^{-1}$
$$
\varphi_{\beta}(r) :=
 \begin{cases}
 \frac{1}{2} b_0 \mu^{-1/\beta} \vert 
 \ln r \vert^{1/\beta} & \text{if } 
 0 < \beta < 1,\\
 \frac{1}{\ln(1 + 2\mu/b_0)} \vert 
 \ln r \vert & \text{if } \beta = 1,\\
 \frac{\beta}{\beta - 1} \big{(} \ln 
 \vert \ln r \vert \big{)}^{-1} \vert 
 \ln r \vert & \text{if } \beta > 1,
 \end{cases}
$$
\big(see \cite[Lemma 3.4]{rage}\big).

\textbf{(A3)} If $U$ 
is compactly supported and if there 
exists $C > 0$ constant such that
on an open non-empty subset of 
$\mathbb{R}^{2}$ $U \geq C$ then
\begin{equation}\label{eq02,4}
\textup{Tr} \hspace{0.4mm} 
\one_{(r,\infty)} \big( p U p \big) 
= \varphi_{\infty}(r) \big( 1 + o(1) 
\big), \hspace{0.2cm} r \searrow 0,
\end{equation}
where 
$\varphi_{\infty}(r) := \big( \ln 
\vert \ln r \vert \big{)}^{-1} 
\vert \ln r \vert, \hspace{0.2cm} 
0 < r < \textup{e}^{-1}$,
\big(see \cite[Lemma 3.5]{rage}\big).

By an evident adaptation of 
\cite[Proof of Corollary 1]{bon} 
we obtain the following corollary 
summarizing useful properties of the 
operator $\mathscr{B}$ defined by 
\eqref{eq4,11}. Therefore we omit 
the proof.

\begin{cor}\label{c4,1}
Let $V$ satisfy assumptions 
\eqref{eq1,12}-\eqref{eq1,13}. Then 
$\mathscr{B} \in \s \big( L^2(\br^3) 
\big)$ and satisfies for $r > 0$ 
small enough 
\begin{equation}
\textup{Tr} \hspace{0.4mm} 
\one_{(r,\infty)} (\mathscr{B}) 
= \mathcal{O} \big( r^{-2/m_\perp} 
\big).
\end{equation}
For $j \in 
\bn^\ast$ the operator-valued 
functions
\begin{equation}\label{eq4,13}
\bc \setminus \big( \mp i[0,+\infty[ 
\big) \ni k \mapsto \mathfrak{B}(k) =
\mathfrak{B}_j^\pm(k) := 
\frac{i\mathscr{B}}{k} \left( I \pm 
\frac{i\mathscr{B}}{k} \right)^{-j} 
\in \s \big( L^2(\br^3) \big)
\end{equation}
are holomorphic and 
\begin{equation}\label{eq4,14}
\Vert \mathfrak{B}(k) \Vert_\spp^p 
\leq f(\theta)^{pj} \sigma_p \big( 
\vert k \vert \big), \qquad p = 1, 
\hspace*{0.1cm} 2,
\end{equation}
where $\theta = \textup{Arg} 
\hspace{1mm} k$, $f(\theta) = \big(
1 - (\sin \theta)_- 
\big)^{-\frac{1}{2}}$ with $s_- := 
\max(-s,0)$ for $s \in \br$ and
\begin{equation}\label{eq4,15}
\sigma_p (r) := 
\left\Vert \frac{\mathscr{B}}{r} 
\left( I + \frac{\mathscr{B}^2}{r^2} 
\right)^{-1/2} \right\Vert_\spp^p = 
\mathcal{O} \big( r^{-2/m_\perp} 
\big), \quad r > 0.
\end{equation}

Further for any $r > 0$ and $p 
\geq 1$
\begin{equation}\label{eq4,16}
2^{-p/2} \Tilde{n}_p (r) \leq 
\sigma_p (r) \leq \Tilde{n}_p (r)
+ \textup{Tr} \hspace{0.4mm} 
\one_{(r,\infty)} (\mathscr{B})
\end{equation}
with
\begin{equation}\label{eq4,17}
\Tilde{n}_p (r) :=
\left\Vert \frac{\mathscr{B}}{r}
\one_{[0,r]}(\mathscr{B}) 
\right\Vert_\spp^p.
\end{equation}

Moreover if the function 
$\textbf{\textup{W}}$ defined by
\eqref{eq2,2} satisfies assumption
\textbf{(A1)} with $m > 2$ then
for $p = 1, 2$ there exists 
constants $C_{m,p}$ and 
$\Tilde{C}_{m,p}$ such that
\begin{equation}\label{eq4,18}
\begin{cases}
\sigma_p (r) = C_{m,p} r^{-2/m} 
\big( 1 + o(1) \big), \\
\Tilde{n}_p (r) = \Tilde{C}_{m,p}
r^{-2/m} 
\big( 1 + o(1) \big),
\end{cases}
\quad r \searrow 0.
\end{equation}
Finally if $\textbf{\textup{W}}$ 
satisfies Assumptions \textbf{(A2)}
then
\begin{equation}\label{eq4,19}
\sigma_p (r) = \varphi_\beta (r) 
\big( 1 + o(1) \big), \quad 
\Tilde{n}_p (r) = o
\big( \varphi_\beta (r) \big), 
\quad r \searrow 0,
\end{equation}
where the functions 
$\varphi_\beta (r)$, $\beta 
\in (0,\infty]$ are defined by 
\eqref{eq02,3} or \eqref{eq02,4}.
\end{cor}

\subsection{Upper bounds on the number of resonances}\label{ss4,1}

The next result concerns an upper 
bound on the number of resonances near 
zero outside a vicinity of $\big\lbrace 
z(k) : k \in -iJ[0,+\infty) \big\rbrace$ 
for potentials $V$ of definite sign 
$J = \pm$. 

\begin{theo}\label{t4,1}
Assume that $V$ satisfying assumptions 
\eqref{eq1,12}-\eqref{eq1,13} is 
of definite sign $J$. Let 
$\mathcal{C}_{\delta}(J)$ be the sector 
defined by 
\begin{equation}
\mathcal{C}_\delta(J) := \big\lbrace 
k \in \bc : - \delta J \im(k) \leq
\vert \re(k) \vert \big\rbrace.
\end{equation}
Then for any $\delta > 0$
there exists $r_0 > 0$ such that for 
any $0 < r <r_0 $
\begin{equation}\label{eq4,20}
\displaystyle \sum_{\substack{z(k) 
\hspace{0.5mm} \in \hspace{0.5mm} 
\textup{Res}(\chh) \\ k 
\hspace{0.5mm} \in \hspace{0.5mm} 
\lbrace r < \vert k \vert < 2r 
\rbrace \cap \mathcal{C}_\delta(J)}} 
\textup{mult} \big( z(k) \big) = 
\mathcal{O} \big( \vert \ln r \vert 
\big).
\end{equation}
\end{theo}

\begin{proof}
Thanks to Proposition \ref{p4,2} for 
$k \in D(0,\epsilon)^\ast$ 
\begin{equation}\label{eq4,21}
\mathcal{T}_{V} \big( z(k) \big) = 
\frac{iJ}{k} 
\mathscr{B} + \mathscr{A}(k),
\end{equation}
where $\mathscr{B}$ is a self-adjoint 
positive operator which does not 
depend on $k$ while $\mathscr{A}(k) 
\in \sd \big( L^{2}(\br^3) \big)$ is 
holomorphic near $k = 0$. Since
$
I + \frac{iJ}{k} \mathscr{B} = 
\frac{iJ}{k} (\mathscr{B} - iJk)
$ then $I + \frac{iJ}{k}\mathscr{B}$ 
is invertible for
$iJk \notin \sigma (\mathscr{B})$ 
and satisfies
\begin{equation}\label{eq4,22}
\small{\left\Vert \left( I + 
\frac{iJ}{k} \mathscr{B} \right)^{-1} 
\right\Vert \leq \frac{\vert k \vert}
{\sqrt{\big( J\im(k) \big)_+^2 
+ \vert \re(k) \vert^2}}}, 
\qquad r_+ := \max (r,0).
\end{equation}
Further it is easy to check that 
for $k \in \mathcal{C}_\delta(J)$ we 
have uniformly with respect to $\vert 
k \vert < r_0 \leq \epsilon$ 
\begin{equation}\label{eq4,23}
\small{\left\Vert \left( I + 
\frac{iJ}{k} \mathscr{B} 
\right)^{-1} \right\Vert \leq 
\sqrt{1 + \delta^{-2}}}.
\end{equation}
Then using \eqref{eq4,21} we can
write
\begin{equation}\label{eq4,24}
\small{I + \mathcal{T}_{V} 
\big( z(k) \big) = \big( I + A(k) 
\big) \left( I + \frac{iJ}{k} 
\mathscr{B} \right)},
\end{equation}
where $A(k)$ is given by 
\begin{equation}\label{eq4,25}
\small{A(k) := \mathscr{A} (k)
\left( I + \frac{iJ}{k} \mathscr{B} 
\right)^{-1}} \in \sd \big( L^{2}
(\br^3) \big).
\end{equation}
Otherwise a simple computation 
allows to obtain
$$
\mathcal{T}_{V} \big( z(k) \big) 
- A(k) = \big( I + A(k) \big) 
\frac{iJ}{k} \mathscr{B} \in 
\mathcal{S}_1 \big( L^{2}(\br^3) 
\big)
$$
since $\mathscr{B} \in 
\mathcal{S}_1 \big( L^{2}(\br^3) 
\big)$ by Corollary \ref{c4,1}. 
Then if we approximate $A(k)$ 
by a finite rank-operator in 
\eqref{eq4,24} and use the formula 
$\textup{det}_2 (I + T) = 
\textup{det} (I + T) 
e^{-\textup{Tr}(T)}$ for $T \in 
\mathcal{S}_1$ we obtain
\begin{equation}\label{eq4,26}
\small{\textup{det}_2 \big( I + 
\mathcal{T}_{V} \big( z(k) \big) 
\big)  = \textup{det} \left( I + 
\frac{iJ}{k} \mathscr{B} \right) 
\times \textup{det}_2 \big( I + 
A(k) \big) e^{-\textup{Tr} \left( 
\mathcal{T}_{V} \big( z(k) \big) - 
A(k) \right)}}.
\end{equation}
Then for $\vert k \vert < r_0$ 
such that $k \in 
\mathcal{C}_\delta(J)$ the zeros 
of $\textup{det}_2 \big( I + 
\mathcal{T}_{V} \big( z(k) \big) 
\big)$ are those of 
$\textup{det}_2 \big( I + A(k) 
\big)$ with the same multiplicities 
thanks to Proposition \ref{p4,1} 
and Property \eqref{eqa,3} applied 
to \eqref{eq4,24}.

Estimate \eqref{eq4,23} and the 
fact that $\mathscr{A}(k)$ is 
holomorphic near $k = 0$ with 
values in $\sd \big( L^{2}(\br^3) 
\big)$ imply that the 
Hilbert-Schmidt norm of $A(k)$ is
uniformly bounded with respect 
to $\vert k \vert < r_0$ small 
enough and $k \in 
\mathcal{C}_\delta(J)$. So we
obtain uniformly with respect 
to $k$
\begin{equation}\label{eq4,27}
\small{\textup{det}_2 \big( I 
+ A(k) \big) = \mathcal{O} 
\left( e^{\mathcal{O} \big( 
\Vert A(k) 
\Vert_{\mathcal{S}_2}^2 \big)} 
\right) = \mathcal{O}(1).}
\end{equation}

In what follows below we prove
a corresponding lower bound of 
\eqref{eq4,27}. Identity 
\eqref{eq4,24} implies that
\begin{equation}\label{eq4,28}
\small{ \big( I + A(k) 
\big)^{-1} = \left( I + 
\frac{iJ}{k} \mathscr{B} \right)
\Big( I + \mathcal{T}_{V} \big( 
z(k) \big) \Big)^{-1}}.
\end{equation}
With the help of \eqref{eq3,21} 
we get for $\im(k^2) > \varsigma 
> 0$
\begin{equation}\label{eq4,281}
\begin{split}
\left\Vert \Big( I + 
\mathcal{T}_{V} \big( z(k) \big) 
\Big)^{-1} \right\Vert & = 
\mathcal{O} \Big( 1 + \left\Vert 
\vert V \vert^{1/2} \big( \chh - 
z(k) \big)^{-1} \vert V 
\vert^{1/2} \right\Vert \Big) \\
& = \mathcal{O} \Big( 1 + 
\big\vert \im(k^2) \big\vert^{-1} 
\Big) = \mathcal{O} \left( 
\varsigma^{-1} \right).
\end{split}
\end{equation}
This together with \eqref{eq4,28} 
yield to
\begin{equation}\label{eq4,29}
\small{\left\Vert \big( I + A(k) 
\big)^{-1} \right\Vert
= \mathcal{O} \big( s^{-1} \big)
\mathcal{O} \big( \varsigma^{-1} 
\big)},
\end{equation}
uniformly with respect to $(k,s)$ 
such that $0 < s < \vert k \vert 
< r_0$ and $\im(k^2) > \varsigma 
> 0$. Let $(\mu_j)_j$ be the 
sequence of eigenvalues of $A(k)$. 
We have
\begin{equation}\label{eq4,30}
\begin{aligned}
\small{\left\vert \Big( 
\textup{det}_2 (I + A(k)) 
\Big)^{-1} \right\vert}
& \small{= \left\vert 
\textup{det} \big( (I + 
A(k))^{-1} e^{A(k)} \big) 
\right\vert} \\
& \small{\leq \prod_{\vert \mu_j 
\vert \leq \frac{1}{2}} \left\vert 
\frac{e^{\mu_j}}{1 + \mu_j} 
\right\vert \times \prod_{\vert 
\mu_j \vert > \frac{1}{2}} 
\frac{e^{\vert \mu_j \vert}}
{\vert 1 + \mu_j \vert}.}
\end{aligned}
\end{equation}
Using the uniform bound $\Vert A(k) 
\Vert_{\mathcal{S}_2} = 
\mathcal{O}(1)$ with respect to 
$\vert k \vert < r_0$ small enough 
and $k \in \mathcal{C}_\delta(J)$
we can prove that the first product
is uniformly bounded. On the other
hand thanks to \eqref{eq4,29} we 
have uniformly with respect to 
$(k,s)$, $0 < s < \vert k \vert 
< r_0$ and $\im(k^2) > \varsigma 
> 0$
\begin{equation}\label{eq4,31}
\small{\vert 1 + \mu_j \vert^{-1}
= \mathcal{O} \big( s^{-1} \big)
\mathcal{O} \big( \varsigma^{-1} 
\big)}.
\end{equation}
Therefore using the fact that the 
second product has a finite number 
of terms $\mu_j$ we deduce from
\eqref{eq4,31} that
\begin{equation}\label{eq4,32}
\small{\left\vert \textup{det}_2 
\big( I + A(k) \big) \right\vert 
\geq C e^{-C \big( \vert \ln 
\varsigma \vert + \vert \ln s 
\vert  \big)},}
\end{equation}
for some $C > 0$ constant. 
To conclude the proof we need the
following Jensen type lemma 
\big(see for instance 
\cite[Lemma 6]{bon}\big):

\begin{lem}\label{la,1}
Let $\Delta$ be a simply connected 
sub-domain of $\mathbb{C}$ and let 
$g$ be a holomorphic function in 
$\Delta$ with continuous extension 
to $\overline{\Delta}$. Assume there 
exists $\lambda_{0} \in \Delta$ such 
that $g(\lambda_{0}) \neq 0$ and 
$g(\lambda) \neq 0$ for $\lambda\in 
\partial \Delta$ the boundary of 
$\Delta$. Let $\lambda_{1}, 
\lambda_{2}, \ldots, \lambda_{N} \in 
\Delta$ be the zeros of $g$ repeated 
according to their multiplicity. Then 
for any domain $\Delta' \Subset \Delta$ 
there exists $C' > 0$ such that 
$N(\Delta',g)$ the number of zeros 
$\lambda_{j}$ of $g$ contained in 
$\Delta'$ satisfies
\begin{equation}\label{eqa,5}
N(\Delta',g) \leq C' \left( 
\int_{\partial \Delta} 
\textup{ln} \vert g(\lambda) \vert 
d\lambda - \textup{ln} \vert 
g(\lambda_{0}) \vert  \right).
\end{equation}
\end{lem}

Consider the domain $\Delta := 
\big\lbrace k \in 
D(0,\epsilon)^\ast : r < \vert k 
\vert < 2r \big\rbrace \cap 
\mathcal{C}_\delta(J)$ with some 
$\im(k_0^2) > \varsigma > 0$, 
$k_0 \in \Delta$. Then
Theorem \ref{t4,1} follows 
immediately by applying the 
Jensen Lemma \ref{la,1} to the 
function $D(\cdot) := \textup{det}_2 
\big( I + A(\cdot) \big)$ on $\Delta$ 
together with Proposition 
\ref{p4,1}, estimates 
\eqref{eq4,27}-\eqref{eq4,32}. 
The proof is complete.
\end{proof}

For general perturbations $V$ 
without sign restriction we have
the following result:

\begin{theo}\label{t4,2}
{\cite[Theorem 2.1]{diom}} 

Let $V$ satisfy assumptions 
\eqref{eq1,12}-\eqref{eq1,13}. 
Then there exists $0 < r_{0} < 
\epsilon$ small enough such 
that for any $0 < r < r_{0}$
\begin{equation}\label{eq4,321}
\displaystyle \sum_{\substack{z(k) 
\hspace{0.5mm} \in \hspace{0.5mm} 
\textup{Res}(\chh) \\ k 
\hspace{0.5mm} \in \hspace{0.5mm} 
\lbrace r < \vert k \vert < 2r 
\rbrace}} \textup{mult} \big( z(k) 
\big) = \mathcal{O} \Big( \textup{Tr} 
\hspace{0.4mm} \one_{(r,\infty)} 
\big( p \textbf{\textup{W}} 
p \big) \vert \ln r \vert \Big).
\end{equation}
\end{theo}

\section{Proof of Theorem $\ref{t2,1}$: Breit-Wigner approximation}\label{s5}


We recall that
$N_{\gamma,\zeta}$ is the constant
defined by \eqref{eq2,51}.

\subsection{Preliminary results}

\begin{lem}\label{l5,1}
Let $V$ satisfy assumptions 
\eqref{eq1,12}-\eqref{eq1,13}
and $\mathcal{T}_V(\cdot)$ be the
operator defined by Lemma 
\eqref{l3,1}. Then on 
$]-N_{\gamma,\zeta}^2,
N_{\gamma,\zeta}^2[ \setminus 
\lbrace 0 \rbrace$
\begin{equation}\label{eq5,1}
\xi' = \xi_2' + \frac{1}{\pi} \im
\textup{Tr} \hspace{0.4mm} \big(
\partial_z \mathcal{T}_V(\cdot) 
\big).
\end{equation}
\end{lem}

\begin{proof}
To get \eqref{eq5,1} thanks to 
\eqref{eq1,23} and \eqref{eq1,25} 
it suffices to prove that for any
function $f \in C_0^\infty \left( 
]-N_{\gamma,\zeta}^2,
N_{\gamma,\zeta}^2[ \setminus 
\lbrace 0 \rbrace \right)$
\begin{equation}\label{eq5,2}
\textup{Tr} \hspace{0.4mm} \left(
\frac{d}{d\varepsilon} 
f(\chho + \varepsilon V)_{\vert 
\varepsilon = 0} \right) = - 
\frac{1}{\pi} \int_\br 
f(\lambda) \im \textup{Tr} 
\hspace{0.4mm} \big( \partial_z 
\mathcal{T}_V(\lambda) \big) 
d\lambda.
\end{equation}
Recall that by the Helffer-Sjöstrand
formula \big(see for instance 
\cite{dima}\big) for an analytic 
extension $\Tilde{f} \in C_0^\infty 
(\br^2)$ of $f$ \big($i.e.$ 
$\Tilde{f}_{\vert\br} = f$ and 
$\Bar{\partial}_z \Tilde{f}
(z) = \mathcal{O} \big( \vert \im(z) 
\vert^\infty \big)$\big) we have
\begin{equation}\label{eq5,3}
f(\chho + \varepsilon V) = - 
\frac{1}{\pi} \int_\bc 
\Bar{\partial}_z \Tilde{f}
(z) (z - \chho - \varepsilon V)^{-1} 
L(dz),
\end{equation}
$L(dz)$ being the Lebesgue measure 
on $\bc$. Quantity \eqref{eq5,3} is 
differentiable with respect to 
$\varepsilon$ and it is easy to
check that
\begin{equation}\label{eq5,4}
\frac{d}{d\varepsilon} 
f(\chho + \varepsilon V)_{\vert 
\varepsilon = 0} = - \frac{1}{\pi} 
\int_\bc \Bar{\partial}_z \Tilde{f}
(z) (z - \chho)^{-1}V(z - \chho)^{-1}
L(dz).
\end{equation}
Exploiting the diamagnetic inequality
and the boundedness of the magnetic
field $b$ it can be checked that 
for $\pm \im(z) > 0$ the operator 
$(z - \chho)^{-1}V(z - \chho)^{-1}$
is of trace class. For 
$\im(z) > 0$ by the cyclicity of 
the trace we have
\begin{equation}\label{eq5,5}
\textup{Tr} \hspace{0.4mm} \Big(
(z - \chho)^{-1}V(z - \chho)^{-1}
\Big) = 
\textup{Tr} \hspace{0.4mm} \Big(
J \vert V \vert^{\frac{1}{2}} 
(z - \chho)^{-2} \vert V 
\vert^{\frac{1}{2}}
\Big) = 
\textup{Tr} \hspace{0.4mm} \Big(
\partial_z \mathcal{T}_V(z) \Big)
\end{equation}
and for $\im(z) < 0$
\begin{equation}\label{eq5,6}
\textup{Tr} \hspace{0.4mm} \Big(
(z - \chho)^{-1}V(z - \chho)^{-1}
\Big) = - \overline{
\textup{Tr} \hspace{0.4mm} \Big(
\partial_z \mathcal{T}_V(\bar{z}) 
\Big)}.
\end{equation}
Therefore the operator 
$\frac{d}{d\varepsilon} 
f(\chho + \varepsilon V)_{\vert 
\varepsilon = 0}$ is of trace 
class and using \eqref{eq5,4} we 
get
\begin{equation}\label{eq5,7}
\begin{split}
\textup{Tr} \hspace{0.4mm} \left(
\frac{d}{d\varepsilon} 
f(\chho + \varepsilon V)_{\vert 
\varepsilon = 0} \right) = - 
\frac{1}{\pi} \int_{\im(z) > 0} &
\Bar{\partial}_z \Tilde{f}
(z) \textup{Tr} 
\hspace{0.4mm} \big( \partial_z 
\mathcal{T}_V(z) \big) L(dz) \\
& + \frac{1}{\pi} \int_{\im(z) < 0} 
\Bar{\partial}_z \Tilde{f}
(z) \overline{ \textup{Tr} 
\hspace{0.4mm} \big( \partial_z 
\mathcal{T}_V(\bar{z}) \big)} L(dz).
\end{split}
\end{equation}
Now \eqref{eq5,2} follows 
immediately from \eqref{eq5,7} using
the Green formula.
\end{proof}

For further use we recall complex 
analysis results due to J. Sjöstrand 
summarized in the following

\begin{prop}\label{p5,0}
{\cite{sj}, \cite{sj1}}

Let $\Omega \subset \bc$ be a simply 
connected domain satisfying 
$\Omega \cap \bc^+ \neq \emptyset$.
Let $z \mapsto F(z,h)$, $0 < h < h_0$
be a family of holomorphic functions
in $\Omega$ having at most a finite
number $N(h) \in \bn^\ast$ of zeros in
$\Omega$. Suppose that
\begin{equation}
F(z,h) = \mathcal{O}(1) 
e^{\mathcal{O}(1)N(h)}, \quad z \in 
\Omega,
\end{equation}
and that there exists constants $C$, 
$\varsigma > 0$ with $\Omega_\varsigma 
:= \big\lbrace z \in \bc : \im(z) > 
\varsigma \big\rbrace \neq \emptyset$ 
such that 
\begin{equation}
\vert F(z,h) \vert \geq e^{-CN(h)},
\quad z \in \Omega_\varsigma.
\end{equation}
Then for any $\Tilde{\Omega} \Subset
\Omega$ there exists $g(\cdot,h)$
holomorphic in $\Omega$ such that
\begin{equation}
F(z,h) = \prod_{j=1}^{N(h)} (z - z_j)
e^{g(z,h)}, \quad \frac{d}{dz} g(z,h) 
= \mathcal{O} \big( N(h) \big), \quad
z \in \Tilde{\Omega},
\end{equation}
where the $z_j$ are the zeros of 
$F(z,h)$ in $\Omega$.
\end{prop}

In the next proposition the domains 
$\mathscr{W}_\pm \Subset \Omega_\pm$ 
and the intervals $I_\pm$ are 
introduced in Section \ref{s2} just 
after \eqref{eq2,51}.

\begin{prop}\label{p5,1}
Assume that $V$ satisfies assumptions 
\eqref{eq1,12}-\eqref{eq1,13}. Let 
$\mathscr{W}_\pm \Subset 
\Omega_\pm$ and $I_\pm$ be as above.
Then there exists $r_0 > 0$ and 
holomorphic functions $g_\pm$ in 
$\Omega_\pm$ satisfying for any 
$\mu \in rI_\pm$
\begin{equation}\label{eq5,8}
\begin{split}
\xi_2'(\mu) = \frac{1}{\pi r} \im 
\hspace{0.5mm} g'_\pm \left( 
\frac{\mu}{r},r \right) + &
\sum_{\substack{w \in \textup{Res}
(\chh) \cap r \Omega_\pm \\ \im (w) 
\neq 0}} \frac{\im (w)}{\pi \vert 
\mu - w \vert^2} \\
& - \sum_{w \in \textup{Res}(\chh) 
\cap r I_\pm} \delta (\mu - w) - 
\frac{1}{\pi} \im \textup{Tr} 
\hspace{0.4mm} \big( \partial_z 
\mathcal{T}_V(\mu) \big),
\end{split}
\end{equation} 
where the functions $g_\pm(\cdot,r)$ 
satisfy
\begin{equation}\label{eq5,9}
\begin{split}
g_\pm(z,r) & = \mathcal{O} \left[ 
\textup{Tr} \hspace{0.4mm} 
\one_{(s_1\sqrt{r},\infty)} 
\big( p \textbf{\textup{W}} p \big)
\vert \ln r \vert + \Tilde{n}_1 
\left( \frac{1}{2} s_1\sqrt{r} 
\right) + \Tilde{n}_2 \left( 
\frac{1}{2} s_1\sqrt{r} \right) 
\right] \\
& = \mathcal{O} \left( \vert 
\ln r \vert r^{-1/m_\perp} \right),
\end{split}
\end{equation}
uniformly with respect to $0 < r 
< r_0$ and $z \in \Omega_\pm$ with 
$\Tilde{n}_q$, $q = 1$, $2$ defined
by \eqref{eq4,17}.
\end{prop}

\begin{proof}
The first step consists to reduce 
the study of the zeros of the 
$2$-regularized perturbation determinant 
to that of a suitable holomorphic function 
in $\Omega_\pm$ satisfying the assumptions 
of Proposition \ref{p5,0}.

By Proposition \ref{p4,2} for $0 < s 
< \vert k \vert \leq s_{0} < \epsilon$
$$
\mathcal{T}_{V} \big( z(k) \big) = 
\frac{iJ}{k} \mathscr{B} + \mathscr{A}
(k).
$$
The operator-valued function $k \mapsto 
\mathscr{A}(k)$ is analytic near zero 
with values in $\sd \big( L^2(\br^3) 
\big)$. Then for $s_{0}$ small enough 
there exists $\mathscr{A}_{0}$ a 
finite-rank operator independent of $k$ 
and $\tilde{\mathscr{A}}(k)$ analytic 
near zero satisfying 
$\Vert \tilde{\mathscr{A}}(k) \Vert < 
\frac{1}{4}$, $\vert 
k \vert < s_{0}$ such that
\begin{equation}
\mathscr{A}(k) = \mathscr{A}_{0} + 
\tilde{\mathscr{A}}(k).
\end{equation}
Consider the decomposition 
\begin{equation}\label{eq5,10}
\mathscr{B} = \mathscr{B} 
\mathbf{1}_{[0,\frac{1}{2}s]} 
(\mathscr{B}) + \mathscr{B} 
\mathbf{1}_{]\frac{1}{2}s,\infty[} 
(\mathscr{B}).
\end{equation}
Obviously $\left\Vert (iJ/k) 
\mathscr{\mathscr{B}} \mathbf{1}_{[0,
\frac{1}{2}s]} (\mathscr{B}) + 
\tilde{\mathscr{A}}(k) 
\right\Vert < \frac{3}{4}$ for $0 < s
 < \vert k \vert < s_{0}$. Then
\begin{equation}\label{eq5,11}
I + \mathcal{T}_{V} \big( z(k) \big) 
= \big( I + \mathscr{K}(k,s) \big) 
\left( I + \frac{iJ}{k} \mathscr{B} 
\mathbf{1}_{[0,\frac{1} {2}s]} 
(\mathscr{B}) + \tilde{\mathscr{A}}
(k) \right),
\end{equation}
where $K(k,s)$ is given by
\begin{equation}\label{eq5,111}
\mathscr{K}(k,s) := \left( \frac{iJ}{k} 
\mathscr{B} \mathbf{1}_{]\frac{1}{2}s,
\infty[} (\mathscr{\mathscr{B}}) + 
\mathscr{A}_{0} \right) \left( I + 
\frac{iJ}{k} \mathscr{B} 
\mathbf{1}_{[0,\frac{1}{2}s]} 
(\mathscr{B}) + \tilde{\mathscr{A}}(k) 
\right)^{-1}.
\end{equation}
Its rank is of order
\begin{equation}
O \left( \textup{Tr} \hspace{0.4mm} 
\one_{(\frac{1}{2}s,\infty)} 
(\mathscr{B}) + 1 \right) = 
\mathcal{O} \left( \textup{Tr} 
\hspace{0.4mm} \one_{(s,\infty)} 
\big( p \textbf{\textup{W}} p \big) 
+ 1 \right)
\end{equation}
according to (\ref{eq4,12}) and 
moreover its norm is bounded by 
$\mathcal{O} \left( s^{-1} \right)$ 
for $0 < s < \vert k \vert < s_{0}$.
Since $\Vert (iJ/k) \mathscr{B} 
\mathbf{1}_{[0,\frac{1}{2}s]} 
(\mathscr{B}) + \tilde{\mathscr{A}}(k)
\Vert < 1$ for $0 < s < \vert k \vert 
< s_{0}$ then 
\begin{equation}
\textup{det} \left( \left( I + \frac{iJ}{k} 
\mathscr{B} \mathbf{1}_{[0,\frac{1}
{2}s]} (\mathscr{B}) + 
\tilde{\mathscr{A}}(k) \right) 
e^{-T_V \big( z(k) \big)} \right) \neq 0.
\end{equation}
Therefore the zeros of 
$\textup{det}_2 \big( I + 
\mathcal{T}_{V} \big( z(k) \big) 
\big)$ are those of 
\begin{equation}\label{eq5,12}
\mathscr{D}(k,s) := \textup{det} 
\big( I + \mathscr{K}(k,s) \big)
\end{equation} 
with the same multiplicities 
thanks to Proposition \ref{p4,1} 
and Property \eqref{eqa,3} applied 
to \eqref{eq5,11}.
The above properties of 
$\mathscr{K}(k,s)$ imply that
\begin{equation}\label{eq5,13}
\begin{aligned}
\mathscr{D}(k,s) & = \prod_{j=1}^
{\mathcal{O} \big( \textup{Tr} 
\hspace{0.4mm} \one_{(s,\infty)} 
(p \textbf{\textup{W}} p) + 1 \big)} 
\big{(} 1 + \lambda_{j}(k,s) \big{)}\\
& = \mathcal{O}(1) \hspace{0.5mm} 
\textup{exp} \hspace{0.5mm} 
\Big( \mathcal{O} \big( \textup{Tr} 
\hspace{0.4mm} \one_{(s,\infty)} 
\big( p \textbf{\textup{W}} p 
\big) + 1 \big) \vert \ln s \vert 
\Big)
\end{aligned}
\end{equation}
for $0 < s < \vert k \vert < s_{0}$, 
where the $\lambda_{j}(k,s)$ are 
the eigenvalues of $\mathscr{K} := 
\mathscr{K}(k,s)$ satisfying 
$\vert \lambda_{j}(k,s) \vert = 
\mathcal{O} \left( s^{-1} \right)$.
If
$\im(k^2) > \varsigma > 0$ with $0 
< s < \vert k \vert < s_{0}$ then
$$
\mathscr{D}(k,s)^{-1} = \det \big( I 
+ \mathscr{K} \big)^{-1} = \det 
\big( I - \mathscr{K} ( I + 
\mathscr{K})^{-1} \big).
$$
Thus with the help of \eqref{eq4,281} 
we can show similarly to \eqref{eq5,13}
that
\begin{equation}\label{eq5,14}
\small{\vert \mathscr{D}(k,s) \vert 
\geq C \hspace{0.5mm} \textup{exp} 
\hspace{0.5mm} \Big( - C \big( 
\textup{Tr} \hspace{0.4mm} \one_{
(s,\infty)} \big( p \textbf{\textup{W}} 
p \big) + 1 \big) \big( \vert \textup{ln} 
\hspace{0.5mm} \varsigma \vert + \vert 
\textup{ln} \hspace{0.5mm} s \vert \big) 
\Big)}.
\end{equation}

Now for $\mathscr{D}(k,s)$ defined by 
\eqref{eq5,12} fix $0 < s_1 < 
\sqrt{\textup{dist} \big( \Omega_\pm,0 
\big)}$ and consider the functions 
\begin{equation}
F_\pm : z \in \Omega_\pm \mapsto 
\mathscr{D} \left( \sqrt{r}\sqrt{z},
\sqrt{r}s_1 \right)
\end{equation}
where
\begin{equation}\label{eq5,15}
\displaystyle \sqrt{z} =
\left\{ \begin{array}{ccc} \sqrt{\rho}
e^{i\frac{\theta}{2}} & \hbox{ if } 
& z = \rho e^{i\theta} \in \Omega_+, \\  
i\sqrt{\rho} e^{-i\frac{\theta}{2}} 
& \hbox{ if } & z = -\rho e^{-i\theta} 
\in \Omega_-. \end{array} \right.
\end{equation}
The functions $F_\pm$ are holomorphic 
in $\Omega_\pm$ and according to 
Proposition \ref{p4,1} $\Tilde{\omega}$ 
is a zero of $F_\pm$ if and only if 
$\omega = \Tilde{\omega}r$ is a 
resonance of $\chh$. Then by Proposition 
\ref{p5,0} applied to $F = F_+$ and 
$F(z) = \overline{F_-(-\bar{z})}$ with 
$h = r$, $N(r) = \textup{Tr} 
\hspace{0.4mm} \one_{(s_1\sqrt{r},
\infty)} \big( p \textbf{\textup{W}} 
p \big) \vert \ln r \vert$ there exists 
holomorphic functions $g_{0,\pm}$ in
$\Omega_\pm$ satisfying for any 
$z \in \Omega_\pm$
\begin{equation}\label{eq5,16}
\mathscr{D}_\pm \left( \sqrt{r}\sqrt{z},
\sqrt{r}s_1 \right) = \prod_{w \in 
\textup{Res}(\chh) \cap r \Omega_\pm} 
\left( \frac{zr - \omega}{r} \right)
e^{g_{0,\pm}(z,r)}
\end{equation}
with
\begin{equation}\label{eq5,17}
\frac{d}{dz} g_{0,\pm}(z,r) 
= \mathcal{O} \left( \textup{Tr} 
\hspace{0.4mm} \one_{(s_1\sqrt{r},
\infty)} \big( p \textbf{\textup{W}} 
p \big) \vert \ln r \vert \right),
\end{equation}
uniformly with respect to $z \in 
\mathscr{W}_\pm$.

From above 
\eqref{eq5,11}-\eqref{eq5,12} we 
know that for $z = z \big( \sqrt{r} 
k \big)$, $0 < s_1 < \vert k \vert 
< s_0$
\begin{equation}\label{eq5,18}
\begin{split}
\textup{det}_2 & \big( I + 
\mathcal{T}_{V} (z) \big) = \\ 
& \mathscr{D} \left( 
\sqrt{r}k, \sqrt{r}s_1 
\right) \textup{det} \left( 
\left( I + \frac{iJ}{\sqrt{r}k} 
\mathscr{B} \mathbf{1}_{[0,\frac{1}
{2}s_1\sqrt{r}]} (\mathscr{B}) + 
\tilde{\mathscr{A}}(\sqrt{r}k) 
\right) e^{-T_V (z)} \right).
\end{split}
\end{equation}
By setting 
$$
\mathfrak{A}(k)
:= \frac{iJ}{\sqrt{r}k} 
\mathscr{B} \mathbf{1}_{[0,\frac{1}
{2}s_1\sqrt{r}]} (\mathscr{B}) + 
\tilde{\mathscr{A}}(\sqrt{r}k)
$$ 
we deduce from \eqref{eq5,11} that 
$\mathcal{T}_{V}(z) - \mathfrak{A}(k)$
is a finite-rank operator thanks to 
the properties of the operator 
$\mathscr{K} \big( \sqrt{r}k,
\sqrt{r}s_1 \big)$ given by 
\eqref{eq5,111}. Then as in 
\eqref{eq4,26} we can prove that
\begin{equation}\label{eq5,19}
\begin{split}
\textup{det} \Big( 
\Big( I + \frac{iJ}{\sqrt{r}k} 
\mathscr{B} \mathbf{1}_{[0,\frac{1}
{2}s_1\sqrt{r}]} & (\mathscr{B}) + 
\tilde{\mathscr{A}}(\sqrt{r}k) 
\Big) e^{-T_V (z)} \Big) \\
& = \textup{det}_2 \big( I + 
\mathfrak{A}(k) \big) 
e^{-\textup{Tr} \big( T_V (z) - 
\mathfrak{A}(k) \big)}
\end{split}
\end{equation}
with $\textup{det}_2 \big( I + 
\mathfrak{A}(k) \big) \neq 0$ since
$\Vert \mathfrak{A}(k) \Vert < 1$ 
for $0 < s_1 < \vert k \vert 
< s_0$. The holomorphicity of
$\tilde{\mathscr{A}}(k)$ with values 
in $\sd \big( L^2(\br^3) \big)$ 
combined with \eqref{eq4,17} of 
Corollary \ref{c4,1} imply that
\begin{equation}\label{eq5,20}
\Vert \mathfrak{A}(k) \Vert_2^2 
= \mathcal{O} \left( \Tilde{n}_2  
\left( \frac{1}{2} \sqrt{r}s_1 
\right) \right).
\end{equation}
Then we have 
\begin{equation}\label{eq5,21}
\textup{det}_2 \big( I + 
\mathfrak{A}(k) \big)
= \mathcal{O}(1) e^{\mathcal{O}(1)
\Tilde{n}_2 \left( \frac{1}{2} 
\sqrt{r}s_1 \right)}.
\end{equation}
On the other hand it can be also 
checked that 
\begin{equation}\label{eq5,22}
\textup{det}_2 \big( I + 
\mathfrak{A}(k) \big)^{-1} = 
\textup{det}_2 \left( I - 
\mathfrak{A}(k) \big( I + 
\mathfrak{A}(k) \big)^{-1} \right) 
= \mathcal{O}(1) e^{\mathcal{O}(1)
\Tilde{n}_2 \left( \frac{1}{2} 
\sqrt{r}s_1 \right)}.
\end{equation}
Then Proposition \ref{p5,0} implies
that there exists $g_1(\cdot,r)$  
holomorphic in $\Omega_\pm$ such 
that 
\begin{equation}\label{eq5,23}
\textup{det}_2 \big( I + 
\mathfrak{A}(k) \big)
= e^{g_1(z,r)}
\end{equation}
with 
\begin{equation}\label{eq5,24}
\frac{d}{dz} g_{1}(z,r) 
= \mathcal{O} \left( \Tilde{n}_2 
\left( \frac{1}{2} \sqrt{r}s_1 
\right) \right),
\end{equation}
uniformly with respect to $z 
\in \mathscr{W}_\pm$. Therefore
according to definition 
\eqref{eq1,24} of $\xi_2$ and by
combining \eqref{eq5,18}, 
\eqref{eq5,16}, \eqref{eq5,19} 
with \eqref{eq5,23} we get
for $\mu = z \big( \sqrt{r}k \big) 
= rk^2 \in r(\Omega_\pm \cap \br)$
\begin{equation}\label{eq5,25}
\begin{split}
\xi_2'(\mu) & = \frac{1}{\pi r} \im 
\hspace{0.5mm} \partial_\lambda 
(g_{0,\pm} + g_1) \left( 
\frac{\mu}{r},r \right) + 
\sum_{\substack{w \in \textup{Res}
(\chh) \cap r \Omega_\pm \\ \im (w) 
\neq 0}} \frac{\im (w)}{\pi \vert 
\mu - w \vert^2} - \sum_{w \in 
\textup{Res}(\chh) \cap r I_\pm} 
\delta (\mu - w) \\
& + \frac{1}{\pi} \im \textup{Tr} 
\hspace{0.4mm} \left( \frac{1}{2k}
\partial_k \left( \frac{iJ}{k} 
\mathscr{B} \mathbf{1}_{[0,\frac{1}
{2}s_1\sqrt{r}]} (\mathscr{B}) + 
\tilde{\mathscr{A}}(k) \right) - 
\partial_z \mathcal{T}_V(\mu + i0) 
\right)
\end{split}
\end{equation} 
with
\begin{equation}\label{eq5,27}
\displaystyle k =
\left\{ \begin{array}{ccc} \sqrt{\mu} 
& \hbox{ if } & \mu > 0, \\  
i\sqrt{-\mu} & \hbox{ if } & \mu < 0. 
\end{array} \right.
\end{equation}
By \eqref{eq4,17} of Corollary 
\ref{c4,1} 
\begin{equation}
\textup{Tr} 
\hspace{0.4mm} \left( \frac{1}{2k}
\partial_k \left( \frac{iJ}{k} 
\mathscr{B} \mathbf{1}_{[0,\frac{1}
{2}s_1\sqrt{r}]} (\mathscr{B}) 
\right) \right) = - 
\frac{iJs_1\sqrt{r}}{4k^3} 
\Tilde{n}_1 \left( \frac{1}{2} 
\sqrt{r}s_1 \right).
\end{equation}
Thanks to Lemma \ref{l3,1} 
$\partial_z \mathcal{T}_V(z)$ is of 
trace class. Then since $\mathscr{B} 
\in \s \big( L^2(\br^3) \big)$ the
operator 
\begin{equation}\label{eq5,271}
\partial_k \tilde{\mathscr{A}}(k) =
\partial_k \mathscr{A}(k) = 
\partial_k \Big( 
\mathcal{T}_V \big( z(k) \big) - 
\frac{iJ}{k} \mathscr{B} \Big)
\end{equation}
is of trace class. Moreover the 
definition \eqref{eq4,111} of 
$\mathscr{A}(k)$ implies that
\begin{equation}\label{eq5,272}
\textup{Tr} \hspace{0.4mm} \left( 
\frac{1}{2k} \partial_k 
\mathscr{A}(k) \right) = 
\textup{Tr} \hspace{0.4mm} 
\left( J \vert V 
\vert^{\frac{1}{2}} \big( \chho 
- k^2 \big)^{-2} \textup{Q} \vert 
V \vert^{\frac{1}{2}} \right) = 
\textup{Tr} \hspace{0.4mm} 
\left( J \vert V 
\vert^{\frac{1}{2}} \big( \chho 
- \mu \big)^{-2} \textup{Q} \vert 
V \vert^{\frac{1}{2}} \right).
\end{equation}
By setting $g_\pm = g_{0,\pm} + 
g_1 + g_2$ with
\begin{equation}
g_2(z) = \frac{iJs_1}{2\sqrt{z}} 
\Tilde{n}_1 \left( \frac{1}{2} 
\sqrt{r}s_1 \right),
\end{equation}
where $\sqrt{z}$ is defined on 
$\Omega_\pm$ by \eqref{eq5,15} we
get the desired conclusion.
\end{proof}

The representation of the SSF near
zero can be specified if the 
potential $V$ is of definite sign 
$J = sign(V)$. According to Remark
\ref{r2,1} in the next proposition
the case $"-"$ is with respect 
the definite sign $J = +$.

\begin{prop}\label{p5,2}
Assume the assumptions of Theorem 
\ref{t2,1} with $V$ of definite
sign $J = sign(V)$. Then for 
$\lambda \in rI_\pm$ \eqref{eq5,8} 
holds with
\begin{equation}\label{eq5,28}
\frac{1}{r} \im \hspace{0.5mm} g'_\pm 
\left( \frac{\lambda}{r},r \right) = 
\frac{1}{r} \im \hspace{0.5mm} \Tilde{g}'_\pm 
\left( \frac{\lambda}{r},r \right) + 
\im \hspace{0.5mm} \Tilde{g}'_{1,\pm}(\lambda)
+ \one_{(0,N_{\gamma,\zeta}^2)}(\lambda) 
J \phi'(\lambda),
\end{equation}
where the function $\phi$ is defined by
\begin{equation}\label{eq5,29}
\phi(\lambda) := \textup{Tr} 
\hspace{0.4mm} \left( \arctan 
\frac{K^\ast K}{\sqrt{\lambda}} \right) 
= \textup{Tr} \hspace{0.4mm} \left( \arctan 
\frac{p\textbf{\textup{W}}p}{2\sqrt{\lambda}}
\right),
\end{equation}
the functions $z \mapsto \Tilde{g}_\pm (z,r)$ 
being holomorphic in $\Omega_\pm$ and satisfying
\begin{equation}\label{eq5,30}
\Tilde{g}_\pm (z,r) = \mathcal{O} 
\big( \vert \ln r \vert \big),
\end{equation}
uniformly with respect to $0 < r < r_0$ and 
$z \in \Omega_\pm$. The functions $z \mapsto 
\Tilde{g}_{1,\pm}(z)$ are holomorphic in 
$\pm ]0,N_{\gamma,\zeta}^2[ 
e^{\pm i]-2\theta_0,2\varepsilon_0[}$
and there exists a positive constant 
$C_{\theta_0}$ depending on $\theta_0$ such 
that
\begin{equation}\label{eq5,31}
\vert \Tilde{g}_{1,\pm}(z) \vert \leq
C_{\theta_0} \sigma_2 \left( \sqrt{\vert 
z \vert} \right)^{\frac{1}{2}}
\end{equation}
for $z \in \pm ]0,N_{\gamma,\zeta}^2[ 
e^{\pm i]-2\theta_0,2\varepsilon_0[}$, 
where the quantity $\sigma_2(\cdot)$ is 
defined by 
\eqref{eq4,15}.
\end{prop}

\begin{proof}
We use notations of Subsection 
\ref{ss4,1}. Hence for $z = z \big( 
\sqrt{r} k \big)$, $0 < s_1 < \vert
k \vert < s_0$ and $k \in 
\mathcal{C}_\delta(J)$ \eqref{eq4,26}
implies that
\begin{equation}\label{eq5,32}
\textup{det}_2 \big( I + 
\mathcal{T}_{V}(z) \big)  = 
\textup{det} \left( I + 
\frac{iJ}{\sqrt{r}k} \mathscr{B} 
\right) \times \textup{det}_2 \big( 
I + A(\sqrt{r}k) \big) 
e^{-\textup{Tr} \left( 
\mathcal{T}_{V}(z) - A(\sqrt{r}k) 
\right)},
\end{equation}
where $A(\sqrt{r}k)$ is given
by \eqref{eq4,25} with $k$ replaced
by $\sqrt{r}k$. Then as in the 
previous proof by applying Proposition
\ref{p5,0} to $\textup{det}_2 \big( 
I + A(\sqrt{r}\sqrt{\cdot}) \big)$
in $\Omega_\pm$ taking into account
\eqref{eq4,27} and \eqref{eq4,32}
we get
\begin{equation}\label{eq5,33}
\textup{det}_2 \big( I + 
A(\sqrt{r}\sqrt{z}) \big) = 
\prod_{w \in \textup{Res}(\chh) 
\cap r \Omega_\pm} \left( 
\frac{zr - \omega}{r} \right)
e^{\Tilde{g}_\pm(z,r)},
\end{equation}
where $\Tilde{g}_\pm$ is holomorphic
in $\Omega_\pm$ such that
\begin{equation}\label{eq5,34}
\frac{d}{dz} \Tilde{g}_\pm(z,r) 
= \mathcal{O} \left( \vert 
\ln r \vert \right),
\end{equation}
uniformly with respect to $z \in 
\mathscr{W}_\pm$. Then
according to definition 
\eqref{eq1,24} of $\xi_2$ and by
combining \eqref{eq5,32}-\eqref{eq5,33} 
we get for $\mu = z \big( \sqrt{r}k \big) 
= rk^2 \in r(\Omega_\pm \cap \br)$
\begin{equation}\label{eq5,35}
\begin{split}
\xi_2'(\mu) & = \frac{1}{\pi r} \im 
\hspace{0.5mm} \partial_\lambda 
\Tilde{g}_\pm \left( 
\frac{\mu}{r},r \right) + 
\sum_{\substack{w \in \textup{Res}
(\chh) \cap r \Omega_\pm \\ \im (w) 
\neq 0}} \frac{\im (w)}{\pi \vert 
\mu - w \vert^2} - \sum_{w \in 
\textup{Res}(\chh) \cap r I_\pm} 
\delta (\mu - w) \\
& + \frac{1}{2k\pi} \im \textup{Tr} 
\hspace{0.4mm} \left( \left( I + 
\frac{iJ}{k} \mathscr{B} \right)^{-1} 
\partial_k \left( \frac{iJ}{k} 
\mathscr{B} \right) \right) - 
\frac{1}{\pi} \im \textup{Tr} 
\hspace{0.4mm} \left( \partial_z 
\mathcal{T}_V(\mu + i0) - 
\frac{1}{2k} \partial_k A(k) \right),
\end{split}
\end{equation} 
where $k$ is defined by \eqref{eq5,27}.

By Lemma \ref{l3,1} $\partial_z 
\mathcal{T}_V(z)$ is of 
trace class. Then as in \eqref{eq5,271}
accordingly to definition \eqref{eq4,25}
of $A(k)$
\begin{equation}\label{eq5,36}
\partial_k A(k) = \partial_k 
\mathscr{A}(k) - \partial_k \left( 
\mathscr{A}(k) \frac{iJ}{k} \mathscr{B} 
\left( I + \frac{iJ}{k} \mathscr{B} 
\right)^{-1} \right) 
\end{equation}
is of trace class. For the first 
term of the RHS of \eqref{eq5,36} 
equality \eqref{eq5,272} holds. For the
second term we have
\begin{equation}\label{eq5,37}
\im \frac{1}{2k} \textup{Tr} 
\hspace{0.4mm} \partial_k \left( 
\mathscr{A}(k) \frac{iJ}{k} \mathscr{B} 
\left( I + \frac{iJ}{k} \mathscr{B} 
\right)^{-1} \right) = \im \frac{1}{2k} 
\partial_k \big( \Tilde{g}_{1,\pm}(k^2) 
\big),
\end{equation}
where $\Tilde{g}_{1,\pm}$ is the
holomorphic function given by
\begin{equation}\label{eq5,38}
\Tilde{g}_{1,\pm}(z) := \textup{Tr} 
\hspace{0.4mm} \left( \mathscr{A}
(\sqrt{z}) \frac{iJ}{\sqrt{z}} 
\mathscr{B} \left( I + 
\frac{iJ}{\sqrt{z}} \mathscr{B} 
\right)^{-1} \right)
\end{equation}
satisfying bound \eqref{eq5,31} by 
Corollary \ref{c4,1}.

For the fourth term of the RHS of
\eqref{eq5,35} we have
\begin{equation}\label{eq5,39}
\begin{split}
\frac{1}{2k} \im \textup{Tr} 
\hspace{0.4mm} & \left( \left( I + 
\frac{iJ}{k} \mathscr{B} \right)^{-1} 
\partial_k \left( \frac{iJ}{k} 
\mathscr{B} \right) \right) \\ 
& = -\frac{1}{2k^2} \im \textup{Tr} 
\hspace{0.4mm} \left( \frac{iJ}{k} 
\mathscr{B} \left( I + \frac{iJ}{k} 
\mathscr{B} \right)^{-1} \right) \\
& \displaystyle =
\left\{ \begin{array}{ccc} 0 
& \hbox{ if } & Jk \in i\br^+, \\  
-\frac{1}{2k^2} \textup{Tr} 
\hspace{0.4mm} \left( \frac{J}{k} 
\mathscr{B} \left( I + 
\frac{\mathscr{B}^2}{k^2} 
\right)^{-1} \right) = J \Phi'(k^2) 
& \hbox{ if } & k \in \br. 
\end{array} \right.
\end{split}
\end{equation} 
Then Proposition \ref{p5,2} follows.
\end{proof}

\subsection{Back to the proof of Theorem \ref{t2,1}}

It follows immediately by 
combining Lemma \ref{l5,1} with 
Propositions \ref{p5,1}-\ref{p5,2}.

\section{Proof of Theorem $\ref{t2,2}$: Singularity at the low ground state}\label{s6}

We begin by applying Theorem 
\ref{t2,1} on intervals of the form 
$r_n[1,2]$, $r_n = 2^n\lambda$ with 
$\lambda > 0$ small enough. Hence for 
$\Omega_+$ a complex neighbourhood 
of $[1,2]$ and $\mu \in r_n[1,2]$ we 
have
\begin{equation}\label{eq5,40}
\begin{split}
\xi'(\mu) = \frac{1}{r_n \pi} \im 
\hspace{0.5mm} \Tilde{g}'_\pm \left( 
\frac{\mu}{r_n},r_n \right) & + 
\sum_{\substack{w \in \textup{Res}(\chh) 
\cap r_n \Omega_+ \\ \im (w) \neq 0}}
\frac{\im (w)}{\pi \vert \mu - w 
\vert^2} \\ 
& - \sum_{w \in \textup{Res}
(\chh) \cap r_n [1,2]} \delta (\mu - w) 
+ \frac{1}{\pi} \left( J \phi' + \im 
\hspace{0.5mm} \Tilde{g}'_{1,\pm} 
\right)(\mu).
\end{split}
\end{equation} 
By Theorem \ref{t4,1} there exists at
most $\mathcal{O} \big( \vert \ln r_n 
\vert \big)$ resonances in 
$r_n\Omega_+$. Then by integrating
\eqref{eq5,40} on $r_n[1,2]$ we 
obtain
\begin{equation}\label{eq5,41}
\xi(r_{n+1}) - \xi(r_n) = 
\frac{1}{\pi} \big[ \im 
\hspace{0.5mm} \Tilde{g}_\pm 
(\cdot,r_n) \big]_1^2 + \mathcal{O} 
\big( \vert \ln r_n \vert \big) 
+ \frac{1}{\pi} \big[ J \phi + \im 
\hspace{0.5mm} \Tilde{g}_{1,\pm} 
\big]_{r_n}^{r_{n+1}}.
\end{equation} 
Now choose $N \in \bn$ such that
$\frac{N_{\gamma,\zeta}^2}{4} \leq
\lambda 2^{N+1} \leq 
\frac{N_{\gamma,\zeta}^2}{2}$. Then
taking the sum in \eqref{eq5,41} 
and exploiting the fact that in 
$\frac{N_{\gamma,\zeta}^2}{2} 
\big[ \frac{1}{2},1 \big]$ the 
functions $\xi$, $\Phi$, 
$\Tilde{g}_{1,\pm}$ are uniformly
bounded together with $\Tilde{g}_\pm 
(\cdot,r_n) = \mathcal{O} \big( 
\vert \ln r_n \vert \big)$ we get
\begin{equation}\label{eq5,42}
\xi(\lambda) = 
\frac{J}{\pi} \Phi (\lambda)
+ \frac{1}{\pi} \im \hspace{0.5mm} 
\Tilde{g}_{1,\pm}(\lambda) +
\sum_{n=0}^N \mathcal{O} \big( \vert 
\ln 2^n \lambda \vert \big) + 
\mathcal{O}(1).
\end{equation} 
Since $N = \mathcal{O} \big( \vert
\ln \lambda \vert \big)$ and 
$\Tilde{g}_{1,\pm}$ satisfies 
\eqref{eq2,12} then \eqref{eq5,42}
implies that for $\lambda$ small
enough
\begin{equation}\label{eq5,43}
\left\vert \xi(\lambda) -
\frac{J}{\pi} \Phi (\lambda)
\right\vert \leq C \vert
\ln \lambda \vert^2 + C \sigma_2 
\left( \sqrt{\lambda} 
\right)^{\frac{1}{2}}
\end{equation} 
for some $C > 0$ constant.
For a Hilbert-Schmidt operator $L$
on $\mathscr{H}$ we have 
$\Vert L \Vert_\sd^2 = \textup{Tr} 
\hspace{0.4mm} (LL^\ast)$. This 
together with the elementary
inequality
$$
\frac{u^2}{1 + u^2} \leq \arctan u
, \quad u \geq 0
$$
imply that $\sigma_2 \left( 
\sqrt{\lambda} \right) \leq \Phi
(\lambda)$, which completes the
proof.

\section{Proof of Theorem $\ref{t2,3}$: Local trace formula}\label{s7}

For simplicity of notation we
ignore in the proof the dependence 
on the subscript $\pm$. Let 
$\Tilde{\psi} \in C_0^{\infty} \big( 
\Omega \big)$ be an almost analytic
extension of $\psi$ such that 
$\Tilde{\psi} = 1$ on $\mathcal{W}$
and 
\begin{equation}
\textup{supp} \hspace*{0.1cm} 
\Bar{\partial}_z \Tilde{\psi} \subset 
\Omega \setminus \mathcal{W}.
\end{equation}
By Applying \eqref{eq1,23} and 
Theorem \ref{t2,1} we get
\begin{equation}\label{eq5,44}
\begin{split}
\textup{Tr} \hspace{0.4mm} & \left[ 
(\psi f) \left( \frac{\chh}{r} \right) 
- (\psi f) \left( \frac{\chho}{r} 
\right) \right] = - \left\langle \xi'
(\lambda),(\psi f) \left( 
\frac{\lambda}{r} \right) \right\rangle 
\\
& = \sum_{w \in \textup{Res}(\chh) \cap 
r \textup{supp} \hspace{0.05cm} \psi} 
(\psi f) \left( \frac{w}{r} \right)
- \frac{1}{\pi} \int (\psi f) \left( 
\frac{\lambda}{r} \right) \im \hspace{0.5mm}
g' \left( \frac{\lambda}{r},r \right) 
\frac{d\lambda}{r} \\
& + \sum_{\substack{w \in \textup{Res}
(\chh) \cap r \textup{supp} 
\hspace{0.05cm} \psi \\ \im (w) 
\neq 0}} \frac{1}{2\pi i} \int (\psi f) 
\left( \frac{\lambda}{r} \right) \left(
\frac{1}{\lambda - \overline{w}} - 
\frac{1}{\lambda - w}\right) d\lambda.
\end{split}
\end{equation}
Using the Green formula and 
\eqref{eq2,8} on $\textup{supp} 
\hspace{0.05cm} \Tilde{\psi}$ we can 
estimate the integral involving $g'$. 
On the other hand for $w \in \bc_- := 
\big\lbrace z \in \bc: \im(z) < 0 
\big\rbrace$ by applying the Green 
formula we get
\begin{equation}
- \frac{1}{\pi} \int_{\bc_-} 
\Bar{\partial}_z \Tilde{\psi}(z) 
\frac{1}{z - w} L(dz) + \Tilde{\psi}(w)
= - \frac{1}{2\pi i} \int_\br 
\Tilde{\psi}(\lambda) \frac{1}{\lambda 
- w} d\lambda
\end{equation}
and 
\begin{equation}
- \frac{1}{\pi} \int_{\bc_-} 
\Bar{\partial}_z \Tilde{\psi}(z) 
\frac{1}{z - \overline{w}} L(dz)
= - \frac{1}{2\pi i} \int_\br 
\Tilde{\psi}(\lambda) \frac{1}{\lambda 
- \overline{w}} d\lambda.
\end{equation}
Since $f$ is holomorphic then with 
the help of the above formulas and 
using the fact that $\Tilde{\psi} 
= \psi$ on $\br$ the third term of 
the RHS of \eqref{eq5,44} is equal 
to
\begin{equation}
\begin{split}
\sum_{\substack{w \in \textup{Res}
(\chh), \im (w) \neq 0}} &
(\Tilde{\psi} f) \left( \frac{w}{r} 
\right) \\ 
& + \sum_{\substack{w \in \textup{Res}
(\chh) \cap r \textup{supp} 
\hspace{0.05cm} \Tilde{\psi} \\ \im 
(w) \neq 0}} \frac{1}{\pi r} 
\int_{\bc_-} (\Bar{\partial}_z 
\Tilde{\psi}) \left( \frac{z}{r} 
\right) f \left( \frac{z}{r} \right) 
\left( \frac{1}{z - \overline{w}} 
- \frac{1}{z - w} \right) L(dz).
\end{split}
\end{equation}
Now by using Theorem \ref{t4,2} in
$\Omega$ and the elementary 
inequality \cite[(5.3)]{pet}
\begin{equation}
\int_{\Omega_1} 
\frac{1}{\vert z - w \vert} L(dz) 
\leq 2 \sqrt{2\pi\text{vol}(\Omega)}
\end{equation}
we get the result.


\section{Appendix}\label{sa}


We recall in this subsection the notion of index (with respect 
to a positively oriented contour) of a holomorphic function and 
a finite meromorphic operator-valued function, see for instance
\cite[Definition 2.1]{bo}.

If a function $f$ is holomorphic in a neighbourhood of a contour 
$\gamma$ its index with respect to $\gamma$ 
is defined by 
\begin{equation}\label{eqa,1}
ind_{\gamma} \hspace{0.5mm} f 
:= \frac{1}{2i\pi} \int_{\gamma} \frac{f'(z)}{f(z)} dz.
\end{equation}
Let us point out that if $f$ is holomorphic in a domain $\Omega$ 
with $\partial \Omega = \gamma$ then thanks to the residues theorem 
$\textup{ind}_{\gamma} \hspace{0.5mm} f$ coincides 
with the number of zeros of $f$ in $\Omega$ taking into account 
their multiplicity. 

Let $D \subseteq \mathbb{C}$ be a connected domain, 
$Z \subset D$ be a pure point and closed 
subset and $A : \overline{D} \backslash Z \longrightarrow 
\textup{GL}(E)$
a be finite meromorphic operator-valued function which is Fredholm 
at each point of $Z$.  The index of $A$ with respect 
to the contour $\partial \Omega$ is defined by 
\begin{equation}\label{eqa,2}
\small{Ind_{\partial \Omega} \hspace{0.5mm} A := 
\frac{1}{2i\pi} \textup{Tr} \int_{\partial \Omega} A'(z)A(z)^{-1} 
dz = \frac{1}{2i\pi} \textup{Tr} \int_{\partial \Omega} A(z)^{-1} 
A'(z) dz}.
\end{equation} 
The following properties are well known: 
\begin{equation}\label{eqa,3}
Ind_{\partial \Omega} \hspace{0.5mm} A_{1} A_{2} = 
Ind_{\partial \Omega} \hspace{0.5mm} A_{1} + 
Ind_{\partial \Omega} \hspace{0.5mm} A_{2};
\end{equation} 
for $K(z)$ a trace class operator
\begin{equation}\label{eqa,4}
Ind_{\partial \Omega} \hspace{0.5mm} (I+K)= 
ind_{\partial \Omega} \hspace{0.5mm} \det \hspace{0.5mm} (I + K).
\end{equation} 
We refer for instance \cite[Chap. 4]{goh} for more details.




\end{document}